\documentclass[11pt, reqno]{amsart}
\usepackage{amsfonts, amssymb}
\usepackage{graphicx}
\usepackage{hyperref}
\usepackage{parskip}
\numberwithin{equation}{section}

\newtheorem{thm}{Theorem}[section]
\newtheorem{corollary}[thm]{Corollary}
\newtheorem{lemma}[thm]{Lemma}

\newtheorem{proposition}[thm]{Proposition}

\theoremstyle{remark}

\theoremstyle{definition}
\newtheorem{definition}[thm]{Definition}

\newcommand{\rr}{{\mathbf{R}^2}}

\newcommand{\dx}{{\; dx}}
\newcommand{\ds}{{\; ds}}

\newcommand{\e}{{\mathbf{e}}}
\newcommand{\bE}{{\mathbf{E}}}
\newcommand{\Ec}{{E_\mathrm{ crit}}}
\newcommand{\ES}{\dot{\mathcal{H}}^1}
\newcommand{\Sn}{{\mathcal{E}}}

\newcommand{\A}{{\mathbf A}}
\newcommand{\B}{{\mathbf B}}
\newcommand{\C}{{\mathbf C}}
\newcommand{\bm}{{\mathbf{m}}}

\newcommand{\tA}{{\widetilde{\mathbf A}}}
\newcommand{\tB}{{\widetilde{\mathbf B}}}
\newcommand{\tC}{{\widetilde{\mathbf C}}}

\newcommand{\T}{{\mathcal T}}
\newcommand{\cM}{{\mathcal M}}
\newcommand{\cR}{{\mathcal R}}
\newcommand{\cL}{{\mathcal L}}
\newcommand{\cS}{{\mathcal S}}

\newcommand{\pd}{{\phi^* \nabla}}
\newcommand{\ph}{{\phi^* h}}

\newcommand{\nline}{{\vspace{\baselineskip}}}

\title[Geometric Renormalization]{Geometric Renormalization Below the Ground State}

\author{Paul Smith}
\date{}
\address{University of California, Berkeley}
\email{smith@math.berkeley.edu}

\begin{document}

\begin{abstract}
The caloric gauge was introduced in \cite{Trenorm} by Tao with studying 
large data energy-critical wave maps
mapping from $\mathbf{R}^{2+1}$ to hyperbolic space 
$\mathbf{H}^m$ in view.
In \cite{BeIoKeTa11} Bejenaru, Ionescu, Kenig, and Tataru
adapted the caloric gauge to the setting of Schr\"odinger maps from $\mathbf{R}^{d + 1}$ to
the standard sphere $S^2 \hookrightarrow \mathbf{R}^3$
with initial data small in the critical Sobolev norm.  
Here we develop the caloric gauge in a bounded geometry
setting with a construction valid
up to the ground state energy.
\end{abstract}

\maketitle

\tableofcontents

\section{Introduction} \label{S:Introduction}
Certain geometric PDE, such as wave maps and Schr\"odinger maps, enjoy
the property of \emph{gauge invariance} at the level of the tangent bundle.
The particular gauge selection one makes plays a crucial role in the analysis
of the PDE in question.  The Coulomb gauge is a well-known, classical
choice well-suited for many problems.  
However, when analyzing
energy-critical wave maps and Schr\"odinger maps, for instance, one finds
that the Coulomb gauge leaves something to be desired.
Tao proposed an alternative (see \cite{Trenorm}), the caloric gauge, which, in particular,
handles
$\mathrm{high} \times \mathrm{high} \to \mathrm{low}$ frequency
interactions much more favorably than does the Coulomb gauge.
Up until now,
the caloric gauge was available only in very restricted settings,
i.e., for target manifolds of constant negative curvature or for
initial data that is suitably small in an appropriate sense.
It is the purpose of this paper to provide a construction of the caloric gauge
valid in a very general setting.

The key tool behind the caloric gauge construction is yet another geometric PDE:
the harmonic map heat flow.  
Harmonic maps and their heat flows have
been intensely studied, and their literature is extensive.
The pioneering work on harmonic map heat flow is due to Eells and Sampson \cite{EeSa64}.
Our work leans most heavily upon the original results of Struwe \cite{St85}, though
we have striven to make this paper largely self-contained in its presentation.
One major exception to this is found in \S \ref{S:Minimal blowup solutions}, where
we cite a major result of \cite{St85}.  The other major exception 
is found in Appendix \ref{Appendix Noncompact}, where
we cite a result of Li and Tam \cite{LiTa91} in order to justify moving from
a compact setting to a noncompact setting.
The contributions this paper makes to the theory of the harmonic
map heat flow are of a technical and analytic nature.
While no new \emph{qualitative} results are shown, several
hard \emph{quantitative} estimates established do appear to be new;
these quantitative estimates play an important role both in the caloric
gauge construction itself and in establishing the appealing quantitative estimates
that the caloric gauge enjoys.

We mention here some of the developments in other directions 
of the theory of harmonic maps and 
harmonic map heat flow.
Results establishing in certain settings the nonexistence of 
smooth nontrivial harmonic maps
are found in Eells-Wood \cite{EeWo76} and  Lemaire \cite{Le78}.
Regularity results originate in Morrey \cite{Mo66}.  See Schoen-Uhlenbeck \cite{ScUh82}
for a compactness property and boundary regularity.
Uniform estimates on derivatives of stationary harmonic maps are derived in
\cite{Li99}.
See Struwe \cite{St85} for global weak solutions and bubbling,
and Topping \cite{To97, To02, To04a, To04b, To04c, To04d}, 
for instance, for sophisticated extensions and refinements along these lines.
For global existence of weak heat flows, see
Chen \cite{Ch89}, Chen-Struwe \cite{ChSt89}.
Recent global heat flow results  for rough initial data
are due to Koch-Lamm \cite{KoLa09}.
See Lin-Wang \cite{LiWa99} for work on convergence of sequences of heat flows.
For the consideration of noncompact target manifolds, see Mitteau \cite{Mi74}, 
Li-Tam \cite{LiTa91}, and the references therein.
The book \cite{LiWa08} by Lin-Wang provides an extensive (though not exhaustive) survey.

Let $\cM$ denote an $m$-dimensional Riemannian manifold with metric $h$
and let $\mathbf{R}^+ := [0, +\infty)$.
A function $\phi: \mathbf{R}^+ \times \rr \to \cM$
satisfying
\begin{equation}
\partial_s \phi(s, x) = (\phi^* \nabla)_j \partial_j \phi(s, x),
\label{heatflowequation}
\end{equation}
where
$\phi^* \nabla$ is the pullback by $\phi$ of the Riemannian connection $\nabla$ on $\cM$,
we call a \emph{harmonic map heat flow}.
Here we parametrize $\mathbf{R}^+ \times \rr$ by $(s, x_1, x_2)$,
interpreting $s$ as a time variable
and implicitly summing
repeated Roman indices over the spatial variables so that $j$ ranges over $1, 2$.  
We usually abbreviate ``harmonic map heat flow'' to just \emph{heat flow}. 
A map defined and satisfying (\ref{heatflowequation}) on
a time subinterval $I$ of $\mathbf{R}^+$ we also call a heat flow,
though sometimes we instead refer to such maps as \emph{partial heat flows},
particularly when we wish to emphasize that a priori such maps do not admit
suitably defined \emph{extensions} from $I$ to all of $\mathbf{R}^+$.

We note that (\ref{heatflowequation}) is invariant with respect to the following
translational and dilational symmetries:
\begin{align}
\phi(s, x) &\mapsto \phi(s, x - x_0) &x_0 \in \rr, \label{trans sym} \\
\phi(s, x) &\mapsto \phi(\frac{s}{\lambda^2}, \frac{x}{\lambda}) &\lambda > 0. \label{dil sym}
\end{align}

We assume throughout that $\cM$ is a smooth manifold and 
has \emph{bounded geometry}, 
meaning that 
the injectivity radius of $\cM$ is bounded
from below and that 
the associated \emph{Riemannian curvature tensor} 
$\cR$ and its covariant derivatives are bounded uniformly on $\cM$
(e.g., see \cite[Chapter 7]{Tr92} for discussion and additional references).  

While the vast majority of our arguments are intrinsic, we shall appeal to extrinsic
arguments in order to establish local existence.  To avoid distracting technicalities,
we therefore make the additional working assumption that $\cM$ is compact
and without boundary (i.e., a closed manifold).  We make implicit use of compactness
in \S \ref{S:Minimal blowup solutions}.
In Appendix \ref{Appendix Noncompact} we show how to remove
the technical compactness assumption and include
a list of all places where we had made use of it;
one motivation for doing so is to capture within
our framework the case where $\cM$ is a hyperbolic space $\mathbf{H}^m$.
As the additional arguments we present
essentially reduce the more general setting to the compact case,
we therefore have a unified approach for constructing the caloric gauge,
valid for all bounded geometry manifolds $\cM$.

By the Nash Embedding
Theorem, $\cM$ admits a smooth isometric embedding $\iota: \cM \hookrightarrow \mathbf{R}^n$
into a Euclidean space of (typically higher) dimension $n$.
Compactness ensures boundedness of the \emph{second fundamental form}
and closely related \emph{shape operator}, which, informally speaking,
describe how serpentine the embedding is.
The extrinsic formulation
of (\ref{heatflowequation}) that $\iota$ gives rise to and that we discuss below
is used to establish local existence and uniqueness for (\ref{heatflowequation}) and to
derive a blowup criterion.  
All other arguments and in fact the blowup criterion itself
admit intrinsic formulations.  Consequently, all of our main conclusions are independent
of the particular value of $n$ or choice of embedding $\iota$.

Let $T \cM$ denote the tangent bundle of $\cM$ and, using the embedding $\iota$, define
$N \cM$ to be its normal bundle in $\mathbf{R}^n$.  
The second fundamental form $\Pi$ of $\cM$ is a
symmetric bilinear form $\Pi : T \cM \times T \cM \to N \cM$ satisfying
\begin{equation*}
\langle \Pi(X, Y), N \rangle = \langle \partial_X N, Y \rangle
\end{equation*}
for vector fields $X, Y$ in $T\cM$ and $N$ in $N \cM$, according to the sign
convention we adopt.
The heat flow equation
(\ref{heatflowequation}) then admits the extrinsic formulation
\begin{equation}
\partial_s \phi = \Delta \phi + \Pi(\phi)(\partial_x \phi, \partial_x \phi),
\label{heatflow ex}
\end{equation}
where here we have identified $\iota \circ \phi$ with $\phi$ just as we shall
do in the sequel when convenient and without further comment.

We assume we are given Cauchy initial data to (\ref{heatflowequation}) 
so that $\phi$ is specified at an initial time $s_0$, which we usually
take to be zero:
\begin{equation}
\phi (0, x) = \phi_0(x).  
\label{initialdata}
\end{equation}
Throughout we require 
that $\phi_0$ be \emph{classical initial data}, 
by which we mean that 
$\iota \circ \phi_0$ of
$\phi_0 \in C_x^\infty(\rr \to \cM)$ differs from a constant 
$\phi_0(\infty) \in \cM \hookrightarrow_\iota \mathbf{R}^n$ by some Schwartz function
$\mathbf{R}^2 \to \mathbf{R}^n$.  That this definition is independent of the choice of
embedding $\iota$  can be checked by examining $\tilde{\iota} \circ \iota^{-1} \circ
(\iota \circ \phi_0)$ for some other choice of smooth isometric embedding $\tilde{\iota}$.
A smooth function $\phi:\rr \to \cM$ differing from some $\phi(\infty) \in \cM$
by a Schwartz function we also call \emph{Schwartz with respect to $\phi(\infty)$.}
Schwartz functions admit an intrinsic definition (entirely equivalent in the case of
compact $\cM$) 
in terms of a base point $p \in \cM$ and the connection $\nabla$, 
though we postpone further discussion of this to
Appendix \ref{Appendix Noncompact}.

Under these assumptions there exists at least for short time
a unique smooth heat flow $\phi$ that satisfies (\ref{heatflowequation}) 
and (\ref{initialdata}) in the classical sense.
In fact, global existence and uniqueness hold in this sense (see \cite{St85}) provided
the initial data $\phi_0$ has energy
\begin{equation*}
E_0 := E(\phi_0) := \frac{1}{2} \int_\rr \lvert \partial_x \phi_0 \rvert_\ph^2 dx
\label{E1 Def}
\end{equation*}
below some $\cM$-dependent threshold which we call $\Ec$.
In particular, $\Ec$ is equal to the least energy that may be carried by a nontrivial finite-energy
harmonic map mapping $\rr \to \cM$ and is set equal to $\infty$ if $\cM$ admits no such maps.
We call this threshold $\Ec$ the \emph{ground state} energy.
That such a threshold exists (or can be taken to be $\infty$)
follows easily from a certain a priori gradient estimate;
see \cite[Chapter 6]{LiWa08} for details.
Note that stationary solutions to (\ref{heatflowequation}) may be identified
with harmonic maps. 
The aptly-named harmonic map heat flow equation
(\ref{heatflowequation}) is the downward $L^2$ gradient
flow associated with the energy functional
\begin{equation}
\phi \mapsto \frac{1}{2} \int_\rr \lvert \partial_x \phi \rvert^2 \dx
\label{E fun}
\end{equation}
and therefore heuristically may be thought to
asymptotically evolve finite energy maps 
into solutions of the
harmonic map equation
\begin{equation*}
(\phi^* \nabla)_j \partial_j \phi(x) = 0.
\end{equation*}

In \S \ref{S:Notation} we review some notation and basic estimates.
The technical heart 
lies in establishing for various energy quantities
associated to
$\phi$ bounds depending only upon the initial energy $E_0$.
As the energy functional given by (\ref{E fun}) is invariant with respect
to the scaling (\ref{dil sym}) enjoyed by (\ref{heatflowequation}),
these bounds are both energy-critical and critical with respect to the heat flow equation.
In addition to recalling some constructions from Riemannian geometry,
in \S \ref{S:Geometry} we prove a set of Bochner-Weitzenb\"ock estimates that play a key role in 
establishing energy bounds, which we then proceed to establish 
in a direct manner for small energies in \S \ref{S:Small Energy Estimates}.
In \S \ref{S:Minimal blowup solutions} 
we use an induction on energy argument, formulated in terms of ``minimal blow-up
solutions'', along with concentration compactness to extend the energy estimates
to all energies below the ground state threshold $\Ec$.  
In particular,
we show that should these estimates not hold for all maps $\phi$ with energy
below $\Ec$, then by exploiting the symmetries (\ref{trans sym}) and (\ref{dil sym})
one may extract a nontrivial harmonic map of energy strictly
less than $\Ec$, thereby contradicting the definition of $\Ec$.  The main estimates
used in extracting the harmonic map and proving convergence originate in 
the work of Struwe and are proven intrinsically 
in \S \ref{S:Energy Estimates}; see \cite{St85, St88}.

In \S \ref{S:Caloric gauge} we introduce the notion of a gauge and in particular
the \emph{caloric gauge}.  The caloric gauge was introduced by Tao in \cite{Trenorm} in the context
of wave maps.  In \cite{TSchroedinger} Tao suggested that the caloric gauge could be effectively applied
to the study of Schr\"odinger maps.
The caloric gauge has been successfully used in the study of global
wellposedness for wave maps and
Schr\"odinger maps: See Tao's application to wave maps in
\cite{T3}--\cite{T7} and the application of Bejenaru, Ionescu, Kenig, and Tataru to
Schr\"odinger maps in \cite{BeIoKeTa11}.  

To see how the caloric gauge compares to the Coulomb gauge, see 
\cite{Trenorm}, \cite{TSchroedinger}
and also \cite[\S 2]{BeIoKeTa11}.  For a general discussion on gauges and a comparison
of several different gauges, see \cite[Chapter 6]{Tdis}.

Up until now, in the presence of positive curvature the caloric gauge
has only been available under the restrictive assumption
of small energy.  
\footnote{For maps mapping from a higher dimensional base space, the
appropriate assumption is small critical Sobolev norm.  See for instance
the application to Schr\"odinger maps in \cite{BeIoKeTa11}.
In this paper we only consider two spatial dimensions.}
We offer a construction valid up to the ground state
energy $\Ec$.
In particular, our main results are summarized as follows:
\begin{itemize}
\item Theorem \ref{Main Energy Estimates}: Given classical initial data
$\phi_0$ with $E(\phi_0) < \Ec$, 
we establish global existence and uniqueness of a smooth heat flow $\phi$
as well as a host of quantitative energy estimates.
\item Theorem \ref{CalGauge EU}: Given classical initial data
$\phi_0$ with $E(\phi_0) < \Ec$, 
we establish existence and uniqueness (up to a choice of boundary frame) 
of the caloric gauge $e \in \Gamma(\mathrm{ Frame}(\phi^* T \cM))$.
In the remainder of \S \ref{S:Caloric gauge}, we use Theorem \ref{Main Energy Estimates}
to prove quantitative estimates on various gauge components.
\item Theorem \ref{Dynamic CalGauge}: We show that
the static caloric gauge construction of Theorem \ref{CalGauge EU} is easily adapted to the
dynamic setting.
\end{itemize}
For applications of these results to the study
of Schr\"odinger maps, we refer the reader to the work of the author in \cite{Sm11, Sm10}.

\section{Notation and basic estimates} \label{S:Notation}
We use $f \lesssim g$ to denote the estimate $\lvert f \rvert \leq C \lvert g \rvert$ for an absolute constant $C > 0$.  To indicate dependence of the constant upon parameters (which,
for instance, can include functions), we use subscripts,
e.g. $f \lesssim_k g$.  As an equivalent alternative we write $f = O(g)$ (or,
with subscripts, $f = O_k(g)$) to denote $\lvert f \rvert \leq C \lvert g \rvert$.

In addition to the usual $L_x^p$ spaces, we employ the Banach spaces $C_x^k$ of
$k$-times continuously differentiable functions, $k = 0, 1, 2, \ldots$, equipped with the norm
\begin{equation*}
\lVert u \rVert_{C_x^k(\rr)} 
:= 
\sup_{0 \leq j \leq k} \sup_{x \in \rr} \lvert \partial_x^j u(x) \rvert
\end{equation*}
and corresponding seminorm
\begin{equation*}
\lVert u \rVert_{\dot{C}_x^k(\rr)}
:=
\sup_{x \in \rr} \lvert \partial_x^k u(x) \rvert,
\end{equation*}
where we use $\partial_x = ( \partial_{x_1}, \partial_{x_2} )$ to denote the gradient operator,
as throughout ``$\nabla$'' will stand for the Riemannian connection on $\cM$.  On occasion
we use the usual Sobolev spaces $W_x^{s, p}$, $H_x^s := W_x^{s, 2}$, their
homogeneous counterparts $\dot{W}_x^{s, p}$, $\dot{H}_x^s$, and
their localized versions $W_{x, \mathrm{ loc}}^{s, p}$, $H_{x, \mathrm{ loc}}^s$.

We record the following special cases of the well-known \emph{Gagliardo-Nirenberg inequality}
for later use:
\begin{align}
\lVert u \rVert_{L_x^\infty(\rr)}
&\lesssim
\lVert u \rVert_{L_x^2(\rr)}^{1/2} \lVert \partial_x^2 u \rVert_{L_x^2(\rr)}^{1/2}
\label{GN2} \\
\lVert u \rVert_{L_x^\infty(\rr)}
&\lesssim
\lVert u \rVert_{L_x^2(\rr)}^{1/3} \lVert \partial_x u \rVert_{L_x^4(\rr)}^{2/3}
\label{GN4} \\
\lVert u \rVert_{L_x^4(\rr)}
&\lesssim
\lVert u \rVert_{L_x^2(\rr)}^{1/2} \lVert \partial_x u \rVert_{L_x^2(\rr)}^{1/2}
\label{GN5}.
\end{align}

Throughout $\Delta := \partial_{x_1}^2 + \partial_{x_2}^2$ stands for the usual
spatial Laplacian on $\rr$.
The heat kernel is defined by
\begin{equation}
e^{s \Delta}u(x) := \frac{1}{4 \pi s} \int_{\rr} e^{-\frac{|x - y|^2}{4s}} u(y) dy.
\label{heat kernel}
\end{equation}
By Young's inequality the parabolic regularity estimate
\begin{equation}
\lVert \partial_x^k e^{s\Delta} u \rVert_{L_x^q(\rr)} \lesssim_{p, q, k}
s^{\frac{1}{q} - \frac{1}{p} - \frac{k}{2}} \lVert u \rVert_{L_x^p(\rr)}
\label{Young}
\end{equation}
holds for all $s > 0$, $k \geq 0$, and $1 \leq p \leq q \leq \infty$.
In particular,
\begin{equation}
\lVert e^{s \Delta} u \rVert_{\dot{C}_x^1(\rr)}
\lesssim
s^{-1/2} \lVert u \rVert_{C_x^0(\rr)}
\end{equation}
and
\begin{equation}
\lVert e^{s \Delta} u \rVert_{C_x^1(\rr)}
\lesssim
(1 + s^{-1/2}) \lVert u \rVert_{C_x^0(\rr)}.
\label{Young2}
\end{equation}
We shall frequently make use of the following more refined variant of (\ref{Young}):
\begin{lemma}(Integrated parabolic regularity).
For any $u \in L_x^2(\rr)$ and $2 < p \leq \infty$ we have
\begin{equation}
\int_0^\infty s^{-2/p} \lVert e^{s\Delta}u \rVert_{L_x^p(\rr)}^2 ds \lesssim_p \lVert u \rVert_{L_x^2(\rr)}^2.
\label{Strichartz}
\end{equation}
\end{lemma}
\begin{proof}
The lemma follows from the $TT^*$ method.  See \cite[Lemma 2.5]{T4} for details.
\end{proof}
Because (\ref{Strichartz}) is analogous to certain Strichartz estimates, 
standard in the study of dispersive PDE, 
what we call ``integrated parabolic regularity" \cite{T4} dubs a 
``parabolic Strichartz estimate".

We also note \emph{Duhamel's formula}
\begin{equation}
u(s_1) = e^{(s_1 - s_0)\Delta}u(s_0) +
\int_{s_0}^{s_1} e^{(s_1 - s)\Delta}(\partial_s u - \Delta u)(s) \ds,
\label{Duhamel}
\end{equation}
valid for any continuous map $s \mapsto u(s)$ from the interval $[s_0, s_1]$ to the space of tempered distributions on $\rr$, which we may take to be either scalar valued or vector valued.

\section{Geometry} \label{S:Geometry}

Recall $\cM$ is an $m$-dimensional Riemannian manifold with metric $h$.  There then exists a 
unique \emph{Riemannian connection} $\nabla$ on $\cM$ satisfying the following 
four characterizing properties.  
Given any smooth vector fields $X, Y, Z$ on $\cM$ and smooth real-valued function $f$ on $\cM$,
the Riemannian connection $\nabla$ is \emph{linear}
\begin{equation}
\nabla_{fX}Y = f \nabla_X Y,
\end{equation}
satisfies the \emph{Leibniz rule}
\begin{equation}
\nabla_X(f Y) = (\partial_X f) Y + f \nabla_X Y,
\label{Leibniz}
\end{equation}
\emph{respects the metric}
\begin{equation}
\partial_X h(Y, Z) = h(\nabla_X Y, Z) + h(Y, \nabla_X Z),
\end{equation}
and is
\emph{torsion-free}
\begin{equation}
\nabla_X Y - \nabla_Y X = [X, Y].
\label{NoTorsion}
\end{equation}
Here $[X, Y]$ denotes the Lie bracket of $X$ and $Y$. 
The Riemannian curvature tensor $\cR$ on $\cM$ is a $(1,3)-$tensor defined by
\begin{equation*}
\cR(X,Y)Z = \nabla_X \nabla_Y Z - \nabla_Y \nabla_X Z - \nabla_{[X, Y]}Z.
\end{equation*}

We caution the reader that oftentimes $\cR$ is defined with the sign convention opposite to the one we adopt.  Via the metric $h$ on $\cM$ the tensor $\cR$ is sometimes viewed as a $(0,4)$--tensor.  
The curvature tensor enjoys several other well-known properties, e.g., symmetries, though
we shall not make explicit use of these in this paper.
A good general introduction and reference is \cite[Chapter 2]{Petersen}.

Recall $\cM$ is isometrically embedded by some $\iota$ in $\mathbf{R}^n$.
By definition,
the pullback via $\iota^*$ of the metric on $\iota(\cM)$
inherited from the ambient space $\mathbf{R}^n$ coincides
with the intrinsic metric $h$.  Compactness implies
that the second fundamental form $\Pi$ and its derivatives are
bounded uniformly.

Let $\phi: \mathbf{R}^d \to \cM$ be a smooth map from the vector space $\mathbf{R}^d$ 
to the manifold $\cM$.  We may then pull back the tangent bundle $T\cM$ by $\phi$ to a 
vector bundle $\phi^*(T\cM)$ over $\mathbf{R}^d$.  A smooth section 
$\gamma: \mathbf{R}^d \to \phi^*(T\cM)$ of this vector bundle is a smooth assignment 
of a tangent vector $\gamma(x) \in T_{\phi(x)}\cM$ to each $x \in \mathbf{R}^d$.  
We may also pull back the connection $\nabla$ on $\cM$ to a 
connection $\phi^* \nabla$ on $\mathbf{R}^d$.  
Let $\partial_1, \ldots, \partial_d$ denote the standard coordinate vector fields on $\mathbf{R}^d$ and
let $\gamma, \gamma^\prime$ be smooth sections of $\phi^*(T\cM)$.
Then from (\ref{Leibniz}) we derive the \emph{Leibniz rule}
\begin{equation}
\partial_i \langle \gamma, \gamma^\prime \rangle_{\phi^* h} = 
\langle (\phi^*\nabla)_i \gamma, \gamma^\prime \rangle_{\phi^*h} +
\langle \gamma, (\phi^* \nabla)_i \gamma^\prime \rangle_{\phi^*h},
\label{Leibniz2}
\end{equation}
and from
(\ref{NoTorsion}) the \emph{zero-torsion property}
\begin{equation}
(\phi^*\nabla)_i \partial_j \gamma = (\phi^*\nabla)_j \partial_i \gamma.
\label{NoTorsion2}
\end{equation}
Note that these constructions readily apply to
$I \times \mathbf{R}^d$ in place of $\mathbf{R}^d$ when we take $I$ to be a connected interval.

For each positive integer $k$ define the energy densities $\e_k$
of a map $\phi$ to be
\begin{align*}
\e_k &:= \lvert (\phi^* \nabla)_x^{k-1} \partial_x \phi \rvert^2_{\phi^*h} \nonumber \\
&:= \langle (\phi^* \nabla)_{j_1} \ldots (\phi^* \nabla)_{j_{k-1}} \partial_{j_k} \phi,
(\phi^* \nabla)_{j_1} \ldots (\phi^* \nabla)_{j_{k-1}} \partial_{j_k} \phi \rangle_{\phi^*h},
\end{align*}
where $j_1, \ldots, j_k$ are summed over $1, 2$.

\begin{lemma}(Bochner-Weitzenb\"ock identities).
Let $\phi$ be a heat flow.  Then for each $k \geq 1$ we have
\begin{equation}
\partial_s \e_k = \Delta \e_k - 2 \e_{k+1} + N_k,
\label{LBW}
\end{equation}
where $N_k$ is a linear combination of $O_k(1)$ terms of the form
\begin{equation*}
T(\phi)(X_1, \ldots, X_j, Y),
\end{equation*}
where $T$ is a tensor with bounded covariant derivatives (bounded on $\cM$),
$j \in \{3, 4, \dots, k+2\}$, 
$Y$ is an iterated spatial covariant derivative of $\phi$ of order $k$,
and each $X_\ell$ is an iterated spatial covariant
derivative of $\phi$ of positive order
so that for each $j$ the sum of orders of derivatives (associated to the $X_\ell$)
equals $k + 2$.
Moreover,
$N_k$ satisfies the bound
\begin{equation}
\lvert N_k \rvert
\lesssim_k
\sum_{j = 3}^{k+2} \sum_{a_1 + \cdots + a_j = k + 2}
(\e_{a_1} \cdots \e_{a_j} \e_k)^{1/2}.
\label{Nk bound}
\end{equation}
\end{lemma}
\begin{proof}
We first prove the $k = 1$ case, which plays a privileged role throughout the sequel.
By the Leibnitz rule (\ref{Leibniz2}) and zero-torsion property (\ref{NoTorsion2}) it holds that
\begin{align}
\partial_s \e_1 &= \partial_s \langle \partial_i \phi, \partial_i \phi
\rangle_{\phi^*h} \nonumber \\
&= 2 \langle (\phi^* \nabla)_s \partial_i \phi, \partial_i \phi
\rangle_{\phi^* h} \nonumber \\
&= 2 \langle (\phi^* \nabla)_i \partial_s \phi, \partial_i \phi
\rangle_{\phi^* h}. \label{privileged case}
\end{align}
Curvature arises in commuting the covariant derivatives so that
\begin{align}
(\pd)_i \partial_s \phi &= (\pd)_i (\pd)_j \partial_j \phi \nonumber \\
&= (\phi^* \nabla)_j (\phi^* \nabla)_i \partial_j \phi 
+ \cR(\phi)(\partial_i \phi, \partial_j \phi) \partial_j \phi \nonumber \\
&= (\phi^* \nabla)_j (\phi^* \nabla)_j \partial_i \phi 
+ \cR(\phi)(\partial_i \phi, \partial_j \phi) \partial_j \phi, \label{se1}
\end{align}
where the last line follows from (\ref{NoTorsion2}).
From (\ref{Leibniz2}) we also obtain
\begin{equation}
\Delta \e_1 = 2 \e_2 + 2
\langle (\phi^* \nabla)_j (\phi^* \nabla)_j \partial_i \phi, \partial_i \phi
\rangle_\ph.
\label{de1}
\end{equation}
It follows from (\ref{privileged case})--(\ref{de1}) that
\begin{equation*}
\partial_s \e_1 = \Delta \e_1 - 2 \e_2
+ \langle \cR(\phi)(\partial_i \phi, \partial_j \phi) \partial_j \phi, \partial_i \phi
\rangle_\ph,
\end{equation*}
which completes the proof of the $k=1$ case in view of the bounded geometry assumption.

For each $k \geq 1$ and each $k$-index 
$\bm = (m_1, \ldots, m_k)$, $m_\ell \in \{1, 2\}$, set
\begin{equation*}
\mathfrak{c}_{s, \bm}^k := (\pd)_s (\pd)_{\bm}^{k} \phi
- (\pd)_{\bm}^k \partial_s \phi
\end{equation*}
and
\begin{equation*}
\mathfrak{c}_{{\bm}, j}^k :=  (\pd)_{\bm}^k (\pd)_j \partial_j \phi - 
(\pd)_j (\pd)_j (\pd)_{\bm}^{k} \phi,
\end{equation*}
where $ (\pd)_{\bm}^{k} \phi$ stands for
\begin{equation*}
(\pd)_{\tilde{\bm}}^{k-1} \partial_{m_k} \phi,
\end{equation*}
with $\tilde{\bm} = (m_1, \ldots, m_{k-1})$ so that $\bm = (\tilde{\bm}, m_k)$.
With an implicit sum over $\{1, 2\}^k$
on repeated multiindices $\bm$,
we have the straightforward relations
\begin{equation}
\partial_s \e_k 
= 2
\langle (\pd)_s (\pd)_{\bm}^{k}\phi,
(\pd)_{\bm}^{k} \phi \rangle_\ph
\end{equation}
and
\begin{equation}
\Delta \e_k = 2 \e_{k+1} + 
2
\langle (\pd)_j (\pd)_j (\pd)_{\bm}^{k} \phi,
(\pd)_{\bm}^{k} \phi \rangle_\ph,
\end{equation}
so that
\begin{equation}
\partial_s \e_k - \Delta \e_k + 2 \e_{k+1} = 
2 \langle \mathfrak{c}_{s, \bm}^k + \mathfrak{c}_{\bm, j}^k, 
(\pd)_{\bm}^{k} \phi \rangle_\ph.
\label{ed}
\end{equation}
Taking now $\tilde{\bm} = (m_2, \ldots, m_k)$ so that $\bm = (m_1, \tilde{\bm})$,
we also have the recursive relations
\begin{equation}
\mathfrak{c}_{s, \bm}^{k} =
\sum_{m_1 = 1,2}
(\pd)_{m_1} \mathfrak{c}_{s, \tilde{\bm}}^{k-1}
- 
\sum_{\bm \in \{1, 2\}^k}
\cR(\phi)(\partial_{m_1} \phi, \partial_s \phi) (\pd)_{\tilde{\bm}}^{k-1} \phi,
\label{dskrec}
\end{equation}
and
\begin{align}
\mathfrak{c}_{\bm, j}^k 
 = \sum_{m_1 = 1,2} (\pd)_{m_1} \mathfrak{c}_{\tilde{\bm}, j}^{k-1} +
\sum_{\bm \in \{1, 2\}^k}& \left[
\cR(\phi)(\partial_{m_1} \phi, \partial_j \phi) (\pd)_j (\pd)_{\tilde{\bm}}^{k-1}
\phi \right. \nonumber \\
&+ \left. (\pd)_j \left( \cR(\phi) (\partial_{m_1} \phi, \partial_j \phi) (\pd)_{\tilde{\bm}}^{k-1}
\phi \right) \right].
\label{dxjrec}
\end{align}
When $k = 1$, 
\begin{equation}
\mathfrak{c}_{s, m_1}^1 = 0
\label{csm1}
\end{equation}
by (\ref{NoTorsion2}).
By the same we also get
\begin{equation}
\mathfrak{c}_{m_1,j}^1 = \cR(\phi)(\partial_{m_1} \phi, \partial_j \phi) \partial_j \phi.
\label{cmj1}
\end{equation}
In view of (\ref{ed}), 
(\ref{csm1}), (\ref{cmj1}), and bounded geometry,
we see that representation (\ref{LBW}) and bound (\ref{Nk bound}) hold
when $k = 1$.
For $k > 1$,  (\ref{LBW}) and (\ref{Nk bound}) follow from induction.  In particular, 
we apply (\ref{dskrec}) and (\ref{dxjrec}) iteratively in
(\ref{ed}),
replacing
$\partial_s \phi$ in (\ref{dskrec}) by $(\phi^* \nabla)_j \partial_j \phi$ courtesy
of (\ref{heatflowequation}).
\end{proof}

We obtain the following inequalities as a corollary.
For convenience we introduce
\begin{equation}
\bE_k(s) := \int_\rr \e_k(s, x) \dx.
\end{equation}
\begin{corollary}
Let $\phi$ be a heat flow.  Then for each $k \geq 1$ we have
\begin{align}
\partial_s \e_k - \Delta \e_k + 2 \e_{k+1} 
&\lesssim_k
\sum_{j = 3}^{k+2} \sum_{a_1 + \cdots + a_j = k + 2}
(\e_{a_1} \cdots \e_{a_j} \e_k)^{1/2}, 
\label{eineq} \\
\partial_s \bE_k + 2 \bE_{k+1}
&\lesssim_k 
\sum_{j = 3}^{k+2} \sum_{a_1 + \cdots + a_j = k + 2}
\int_\rr (\e_{a_1} \cdots \e_{a_j} \e_k)^{1/2} \dx, \label{Eineq} \\
\bE_k(s) -\bE_k(r) + &2\int_{r}^{s} \bE_{k+1}(s^\prime) ds^\prime \nonumber \\
&\lesssim_k
\sum_{j=3}^{k+2} \sum_{a_1 + \cdots + a_j = k + 2} 
\int_{r}^s \int_\rr (\e_{a_1} \cdots \e_{a_j} \e_k)^{1/2} dx \; ds^\prime, 
\label{intEineq}
\end{align}
for $0 \leq r \leq s$.
\end{corollary}
\begin{proof}
Inequality (\ref{eineq}) follows from (\ref{LBW}) and (\ref{Nk bound}).
The remaining inequalities follow from integrating in (\ref{eineq}).
\end{proof}

The following corollary and its proof are essentially contained in \cite[Corollary 3.6]{T4}.
We present the statement and proof here for convenience
and to introduce the diamagnetic inequality,
which reappears in subsequent arguments.
\begin{corollary}
Let $\phi$ be a heat flow.  Then for each $k \geq 1$ we have
\begin{equation}
\partial_s \sqrt{\e_k} - \Delta \sqrt{\e_k}
\leq C_k
\sum_{j = 3}^{k+2} \sum_{a_1 + \cdots + a_j = k + 2}
(\e_{a_1} \cdots \e_{a_j})^{1/2}
\label{sqrty}
\end{equation}
in the distributional sense for some constant $C_k > 0$.
\end{corollary}
\begin{proof}
We shall only work formally; our arguments may be justified by replacing $\sqrt{\e_k}$ by $\sqrt{\epsilon^2 + \e_k}$ and taking distributional limits sending $\epsilon \to 0$. 
Writing
\begin{equation*}
\partial_s \e_k = 2 \sqrt{\e_k} \partial_s \sqrt{\e_k}
\quad\text{and}\quad
\Delta \e_k = 2 \sqrt{\e_k} \Delta \sqrt{\e_k} + 2 \lvert \partial_x \sqrt{\e_k} \rvert^2,
\end{equation*}
we have
\begin{equation}
\partial_s \sqrt{\e_k} - \Delta \sqrt{\e_k} + \frac{\e_{k+1} - \lvert \partial_x \sqrt{\e_k}\rvert^2}{\sqrt{\e_k}}
\lesssim_k
\sum_{j = 3}^{k+2} \sum_{a_1 + \cdots + a_j = k + 2}
(\e_{a_1} \cdots \e_{a_j})^{1/2}.
\label{presqrty}
\end{equation}
The \emph{diamagnetic inequality}
\begin{equation}
\lvert \partial_x \sqrt{\e_k} \rvert \leq \sqrt{\e_{k+1}}
\label{diamagnetic}
\end{equation}
follows from the Leibniz rule (\ref{Leibniz2}) and Cauchy-Schwarz:
\begin{equation*}
\lvert \partial_x \e_k \rvert \leq 2 \sqrt{\e_k} \sqrt{\e_{k+1}}.
\end{equation*}
Using (\ref{diamagnetic}) in (\ref{presqrty}) implies (\ref{sqrty}).
\end{proof}
Using positivity of the heat kernel (\ref{heat kernel}) (or the maximum principle), 
we derive from (\ref{sqrty}) a Duhamel-type \emph{inequality} (cf.~(\ref{Duhamel}))
\begin{equation}
\sqrt{\e_k}
\leq
e^{(s_1 - s_0)\Delta} \sqrt{\e_k}(s_0)
+
C_k \int_{s_0}^{s_1} e^{(s_1 - s)\Delta}
\sum_{j = 3}^{k+2} \sum_{a_1 + \cdots + a_j = k + 2}
(\e_{a_1} \cdots \e_{a_j})^{1/2}(s) \ds.
\label{Duhamel inequality}
\end{equation}

\section{Local theory and continuation} \label{S:Small Energy Estimates}

We begin with a standard local existence and uniqueness result.

\begin{thm}
Let $\phi_0 : \rr \to \cM \hookrightarrow_\iota \mathbf{R}^n$ differ from $\phi(\infty) \in \cM$
by a Schwartz function.
Then
there exists an $S > 0$ and a unique smooth heat flow extension
$\phi : [0, S] \times \rr \to \cM \hookrightarrow_\iota \mathbf{R}^n$
such that for each fixed
$s \in S$ the function $\phi(s) - \phi(\infty)$ is Schwartz.
Moreover, $\phi$ may be smoothly continued in time provided $\partial_x \phi$
remains bounded.
\label{LWP}
\end{thm}
\begin{proof}
Rewrite (\ref{heatflow ex}) using the Duhamel formula (\ref{Duhamel}) as
\begin{equation}
\phi(s) = e^{s \Delta} \phi(0) +
\int_0^s e^{(s - s^\prime) \Delta}
\left(
\Pi(\phi)(\partial_x \phi, \partial_x \phi)
\right)(s^\prime) \ds^\prime.
\label{heatflow Duhamel}
\end{equation}
As $\phi(0) = \phi_0$ differs from a Schwartz function by a constant, the linear
solution $e^{s \Delta} \phi_0$ has all derivatives uniformly bounded.
Using (\ref{Young2}) and Picard iteration establishes local existence and uniqueness in
the space $C_s^0 C_x^1 ([0, S] \times \rr)$ for sufficiently small $S$.
In view of the fact that $\iota$ is a uniform isometric embedding,
differentiating (\ref{heatflow Duhamel}) and using the higher order parabolic regularity
estimates (\ref{Young}) leads to the conclusion that $\phi$ is smooth and has
all derivatives bounded on $[0, S]$.  By working in weighted spaces, e.g.
$\lVert \langle x \rangle^k \rVert_{C_s^0 C_x^\ell}([0, S] \times \rr)$, one can
show that $\phi(s) - \phi(\infty)$ is rapidly decreasing in space for each fixed $s$,
which, combined with boundedness of derivatives and Taylor's theorem with remainder
ensures that $\phi(s)$ is Schwartz.

From the usual iteration argument and the compactness assumption
it follows that $\phi$ may be continued in time and moreover remains Schwartz
so long as $\partial_x \phi$ remains bounded.
\end{proof}

\subsection{Small energy results}\label{S:Small energy}

The purpose of this subsection is to establish the following
\begin{thm}
Let $\phi_0 : \rr \to \cM$ be classical initial data.
Then for $E_0 := E(\phi_0)$ sufficiently small
there exists a unique smooth heat flow extension
$\phi: \mathbf{R}^+ \times \rr \to \cM$ such that
$\phi(s) \to_{s \to \infty} \phi(\infty)$ in the $C_x^\infty(\rr \to \cM)$ topology
and such that for $k \geq 1$ the heat flow extension
$\phi$ satisfies the following energy estimates:
\begin{align}
\int_0^\infty \int_\rr s^{k-1} \e_{k+1}(s, x) \dx ds
&\lesssim_k E_0, \label{ee1} \\
\sup_{0 < s < \infty} s^{k-1} \int_\rr \e_k (s, x) \dx
&\lesssim_k E_0, \label{ee2} \\
\sup_{ \substack{ 0 < s < \infty \\ x \in \rr } } s^k \e_k(s, x)
&\lesssim_k E_0, \label{ee3} \\
\int_0^\infty s^{k-1} \sup_{x \in \rr} \e_k(s, x) \ds
&\lesssim_k E_0 \label{ee4}.
\end{align}
\label{SmallEnergyResult}
\end{thm}

Our first step towards proving the theorem is to
establish some a priori estimates on energy densities 
under the assumption that the initial data has sufficiently small energy.
The proofs resemble closely 
arguments found in \cite[\S4]{T4}.  
The present setting, however, demands several adaptations and also some additions,
e.g. the proof of Proposition \ref{Newprop}, which has no counterpart in \cite{T4}.
After the a priori estimates are established, the global existence, uniqueness, and smoothness
claims of the theorem follow easily from a continuity argument.  The convergence claim
and the family of estimates (\ref{ee4}) will be taken up last and will follow respectively
from (\ref{ee3}) and (\ref{ee1}).

Let $I$ be a connected time interval containing the subinterval $[0, r]$, and let $\phi : I \times \rr \to \cM$ be a heat flow with initial energy $E_0$ at time $s = 0$.
Assume that on $[0, r]$ the covariant parabolic regularity estimates
\begin{align*}
\int_0^{r} \int_\rr s^{k-1} \e_{k+1}(s, x) \dx ds &\leq \sqrt{E_0}, \\
\sup_{0 < s < r} s^{k-1} \int_\rr \e_k(s, x) \dx &\leq \sqrt{E_0}, \\
\sup_{ \substack{ 0 < s < r \\ x \in \rr}} s^k \e_k(s, x) &\leq \sqrt{E_0}
\end{align*}
hold for $k = 1, 2, \ldots, N$, with, say, $N = 10$.  
For each such $k$ we respectively label these statements $\A_k, \B_k,$ and $\C_k$.  
Let $\tA_k, \tB_k$ and $\tC_k$ respectively refer to the similar but stronger 
(taking $E_0$ small) statements
\begin{align}
\int_0^{r} \int_\rr s^{k-1} \e_{k+1}(s, x) \dx ds & \lesssim_k E_0, \\ 
\sup_{0 < s < r} s^{k-1} \int_\rr \e_k(s, x) \dx & \lesssim_k E_0, \\ 
\sup_{ \substack{ 0 < s < r \\ x \in \rr}} s^k \e_k(s, x) &\lesssim_k E_0, 
\end{align}
where the implied constants are only allowed to depend upon the integer $k$ and the topological and geometric properties of $\cM$.
Note that $\A_k$, $\B_k$, and $\C_k$, $k = 1, \ldots, N$ may always be made to hold simultaneously
by taking $r$ sufficiently small.
We therefore assume now and throughout this subsection that this has been done.
In addition to the above inequalities we shall need the integrated $L_x^\infty$ parabolic regularity estimate
\begin{equation}
\int_0^{r} \lVert \e_1(s) \rVert_{L_x^\infty(\rr)} \ds \leq 1,
\label{Lia}
\end{equation}
and so we choose $r$ even smaller if necessary so that this holds as well.
We shall improve (\ref{Lia}) to
\begin{equation}
\int_0^{r} \lVert \e_1(s) \rVert_{L_x^\infty(\rr)} \ds \lesssim E_0.
\label{Li}
\end{equation}
From (\ref{Eineq}) and symmetry it follows that
\begin{equation*}
\partial_s \bE_k  + 2 \bE_{k+1} \lesssim_k
\sum_{j=3}^{k+2} 
\sum_{ \substack{ a_1 + \cdots + a_j = k + 2 \\ a_1 \leq \cdots \leq a_j } }
\int_\rr (\e_{a_1} \cdots \e_{a_j} \e_k )^{1/2} dx.
\end{equation*}
We show that together $\A_k, \B_k, \C_k$, $k = 1, \ldots N$, and (\ref{Lia})
imply the a priori estimates $\tA_k, \tB_k, \tC_k$, $k = 1, \ldots, N$, and (\ref{Li}).
This we prove through a sequence of propositions:
Proposition \ref{EnergyMono} establishes that $\tB_1$ is true.
In Propositions \ref{tB49tA} and \ref{Newprop} respectively, 
$\tB_1$ and (\ref{Lia}) are used to show that
$\A_1$ and (\ref{Li}) hold.
By using $\tA_1$,
Proposition \ref{A1 Implies Ak} concludes that $\tA_k$, $k = 2, \ldots, N$ hold.
The $\tA_k$'s then imply $\tB_j$, $j =1, \ldots, N$ according to Proposition \ref{Aks Imply Bk}.
Finally, the $\tB_k$'s are used to establish $\tC_j$, $j = 1, \ldots, N$ in Proposition \ref{B to the C}.

Assume now that $k > 1$ so that $a_j > 1$.
Using the hypotheses $\C_j$ for $j < k$, we  conclude
\begin{equation*}
2 \bE_{k+1} + \partial_s \bE_k \lesssim_k \sqrt{E_0} 
\sum_{a=2}^k \int_\rr s^{-(k+2-a)/2} \e_a^{1/2} \e_k^{1/2} \dx.
\end{equation*}
Applying Cauchy-Schwarz yields
\begin{equation*}
2 \bE_{k+1} + \partial_s \bE_k \lesssim_k \sqrt{E_0} \sum_{a=2}^k s^{-(k+2-a)/2}
\bE_a^{1/2} \bE_k^{1/2},
\end{equation*}
and so by the arithmetic- geometric-mean inequality,
\begin{equation}
2 \bE_{k+1} + \partial_s \bE_k \lesssim_k \sqrt{E_0} \sum_{a=2}^k s^{a-k-1} \bE_a.
\label{AGM}
\end{equation}
Inequality (\ref{AGM}) will be used extensively in the proofs below.

\begin{proposition}
Suppose that $\tA_1$ holds.  Then $\tA_k$ holds for $k = 2, 3, \ldots, N$.
\label{A1 Implies Ak}
\end{proposition}
\begin{proof}
Let $\rho$ be a smooth nonnegative cutoff function supported on $[1/2, 2]$ and bounded by $1$
and let $s_0 \in (0 , r/2)$.
From (\ref{AGM}) it follows that
\begin{align*}
s_0^{k-1} \int_{0}^{r} \rho(\frac{s}{s_0}) \bE_{k+1} ds
- &s_0^{k-2} \int_{0}^{r} (\partial_s \rho)(\frac{s}{s_0}) \bE_k(s) ds \\
&\lesssim_k
\sqrt{E_0} \sum_{a = 2}^k s_0^{a-2}  \int_{0}^{r} \rho(\frac{s}{s_0}) \bE_a(s) ds,
\end{align*}
which by the triangle inequality may be rewritten as
\begin{align*}
s_0^{k-1} \int_{0}^{r} \rho(\frac{s}{s_0}) \bE_{k+1} ds
\lesssim_k& \;
\sqrt{E_0} \sum_{a = 2}^k s_0^{a-2}  \int_{0}^{r} \rho(\frac{s}{s_0}) \bE_a(s) ds\\
&+
\left\lvert
s_0^{k-2} \int_{0}^{r} (\partial_s \rho)(\frac{s}{s_0}) \bE_k(s) ds
\right\rvert.
\end{align*}
Using induction, apply $\A_1$, $\A_2$, \ldots, $\A_{k-2}$ to respectively bound the integrals of 
$\bE_2, \bE_3, \ldots, \bE_{k-1}$ in the right hand side
hand side and use $\tA_{k-1}$ to bound the integral of $\bE_k$.  
Summing dyadically in $s_0$ completes the proof.
\end{proof}

\begin{proposition}
Let $k > 1$ and assume that $\tA_{k-1}$ and $\tA_k$ hold.  Then $\tB_k$ holds.
\label{Aks Imply Bk}
\end{proposition}
\begin{proof}
From (\ref{AGM}) it follows that for any $s_0$ in the range $0 < s_0 < r$ we have
\begin{equation*}
2 \int_0^{s_0} s^{k-1} \bE_{k+1} \ds - \int_0^{s_0} (k-1)s^{k-2} \bE_k \ds + s_0^{k-1}
\bE_k(s_0) \lesssim_k \sqrt{E_0} \sum_{a=2}^k \int_0^{s_0} s^{a-2} \bE_a \ds.
\end{equation*}
Applying hypotheses $\A_1, \A_2, \ldots, \A_{k-1}$ to control
the integrals on the right hand side 
and $\tA_{k-1}$ and $\tA_k$ to control those on the left
implies by the triangle inequality that
\begin{equation*}
{s_0}^{k-1} \bE_k(s_0) \lesssim_k E_0,
\end{equation*}
thereby establishing $\tB_k$.
\end{proof}

\begin{proposition}
Inequality $\tB_1$ holds.
\label{EnergyMono}
\end{proposition}
\begin{proof}
This follows from the well-known property of \emph{energy monotonicity}
(see for instance \cite{St85}),
established using the observation that
the heat flow equation is the $L^2$-gradient flow associated to the energy functional
(\ref{E fun}):
\begin{align*}
\int_\rr \langle \partial_s \phi, \partial_s \phi \rangle_\ph \dx
&= \int_\rr \langle (\pd)_j \partial_j \phi, \partial_s \phi \rangle_\ph \dx \\
&= - \int_\rr \langle \partial_j \phi, (\pd)_j \partial_s \phi \rangle_\ph \dx \\
&= - \int_\rr \langle \partial_j \phi, (\pd)_s \partial_j \phi \rangle_\ph \dx \\
&= - \frac{1}{2} \partial_s \int_\rr \e_1 \dx.
\end{align*}
\end{proof}

\begin{proposition}
Inequality $\tB_1$ implies $\tC_1$.  For $k > 1$ the inequalities 
$\tB_1, \tB_2, \ldots, \tB_k$ and $\tC_1, \tC_2, \ldots, \tC_{k-1}$ together imply $\tC_k$.
\label{B to the C}
\end{proposition}
\begin{proof}
Applying Duhamel (\ref{Duhamel inequality}) and parabolic regularity (\ref{Young}) to (\ref{sqrty}), we get
\begin{align*}
\lVert \sqrt{\e_k}(s_0) \rVert_{L_x^\infty} 
\lesssim_k& \;
s_0^{-1/2} 
\lVert \sqrt{\e_k}(s_0/2) \rVert_{L_x^2} \nonumber \\
&+ \int_{s_0/2}^{s_0} 
\sum_{j = 3}^{k+2}
\sum_{a_1 + \cdots + a_j = k + 2}
(s_0-s)^{-1/2} \lVert (\e_{a_1} \cdots \e_{a_j})^{1/2} \rVert_{L_x^2} \ds.
\end{align*}
To bound the first term on the right hand side, use $\tB_k$.  For the remaining terms, use the hypotheses $\tB_j$, $\tC_j$.
\end{proof}

\begin{proposition}
Inequalities $\tB_1$ and \textrm{ (\ref{Lia})} imply $\tA_1$.
\label{tB49tA}
\end{proposition}
\begin{proof}
Here we apply (\ref{intEineq}) with $k = 1$.
By $\tB_1$ and the triangle inequality, we need only control the right hand side by $E_0$:
\begin{align*}
\int_0^{r} \int_\rr \e_1^2(s, x) \dx ds &\leq 
\int_0^{r} \sup_x \e_1(s, x) \int_\rr \e_1(s, x) \dx ds \\
&\leq 2E(s) \int_0^{r} \sup_x \e_1(s, x) \ds \\
& \lesssim E_0,
\end{align*}
where the last inequality follows from $\tB_1$ and (\ref{Lia}).
\end{proof}

\begin{proposition}
Inequalities $\tB_1$ and \textrm{ (\ref{Lia})} imply \textrm{ (\ref{Li})}.
\label{Newprop}
\end{proposition}
\begin{proof}
In view of (\ref{sqrty}) with $k = 1$
we apply Duhamel (\ref{Duhamel inequality}) 
to $\sqrt{\e_1}$ and obtain
\begin{equation*}
\sqrt{\e_1}(s, x) \leq e^{s \Delta} \sqrt{\e_1}(0, x) + 
C_1 \int_0^s e^{(s-s^\prime) \Delta} \sqrt{\e_1}^3(s^\prime, x) \; ds^\prime.
\end{equation*}
Taking $L_x^\infty$ norms and squaring gives
\begin{equation}
\lVert \e_1(s) \rVert_{L_x^\infty} \lesssim \lVert e^{s \Delta} \sqrt{\e_1}(0) \rVert_{L_x^\infty}^2 +
\left\lVert \int_{0}^s e^{(s-s^\prime) \Delta} \sqrt{\e_1}^3(s^\prime, x) \; ds^\prime 
\right\rVert_{L_x^\infty}^2.
\label{Lis}
\end{equation}
We now integrate in $s$ and bound each of the terms of the right hand side separately.
Using the $p = \infty$ case of 
(\ref{Strichartz}) yields
\begin{equation*}
\int_{0}^{r} \lVert e^{s \Delta} \sqrt{\e_1}(0) \rVert_{L_x^\infty}^2 ds \lesssim \lVert \sqrt{\e_1}(0)\rVert_{L_x^2}^2 = 2E_0
\end{equation*}
independently of $r$.
By two applications of Minkowski's inequality, we have for the remaining term that
\begin{align*}
I_2 &:= \int_{0}^{r} \left\lVert \int_{0}^s e^{(s - s^\prime)\Delta} \sqrt{\e_1}^3(s^\prime, x) 
ds^\prime \right\rVert_{L_x^\infty}^2 ds \\
&\leq \int_{0}^{r} \left( \int_{0}^{r} \lVert e^{(s - s^\prime)\Delta}\sqrt{\e_1}^3(s^\prime, x) \rVert_{L_x^\infty} ds^\prime \right)^2 ds \\
&\leq \left( \int_{0}^{r} \left( \int_{0}^{r} \lVert e^{(s - s^\prime)\Delta} \sqrt{\e_1}^3(s^\prime, x) \rVert_{L_x^\infty}^2 ds
\right)^{\frac{1}{2}} ds^\prime \right)^2
\end{align*}
By (\ref{Strichartz}) with $p = \infty$, it follows that
\begin{equation*}
\int_{0}^{r} \lVert e^{(s - s^\prime)\Delta} \sqrt{\e_1}^3(s^\prime, x) \rVert_{L_x^\infty}^2 ds
\lesssim \lVert \sqrt{\e_1}^3(s^\prime)\rVert_{L_x^2}^2.
\end{equation*}
Thus
\begin{align*}
I_2 &\lesssim \left( \int_{0}^{r} \lVert \sqrt{\e_1}^3(s^\prime)\rVert_{L_x^2} ds^\prime \right)^2 \\
&\leq \left( \int_0^{r} \lVert \sqrt{\e_1}(s^\prime) \rVert_{L_x^\infty}^2 
\lVert \sqrt{\e_1}(s^\prime)\rVert_{L_x^2} ds^\prime \right)^2 \\
&\leq \sup_s \lVert \sqrt{\e_1}(s) \rVert_{L_x^2}^2 \cdot \left( \int_0^{r} \lVert \e_1(s^\prime)
\rVert_{L_x^\infty} ds^\prime\right)^2 \\
&\leq 2E_0 \left( \int_0^{r} \lVert \e_1(s^\prime)\rVert_{L_x^\infty} ds^\prime\right)^2,
\end{align*}
where the last inequality holds because of $\tB_1$.
Therefore
\begin{equation*}
\int_0^{r} \lVert \e_1(s) \rVert_{L_x^\infty} ds \lesssim E_0 + E_0 
\left( \int_0^{r} \lVert \e_1(s) \rVert_{L_x^\infty} ds \right)^2,
\end{equation*}
and so in view of (\ref{Lia}) it follows that for $E_0$ sufficiently small
\begin{equation*}
\int_0^{r} \lVert \e_1(s) \rVert_{L_x^\infty} ds \lesssim E_0
\end{equation*}
independently of $r$.
\end{proof}

\begin{proof}(Proof of Theorem \ref{SmallEnergyResult}).
Via a straightforward continuity argument, Theorem \ref{LWP} together with 
Propositions  \ref{A1 Implies Ak}--\ref{Newprop} imply the
global existence and uniqueness claims of Theorem \ref{SmallEnergyResult}
along with (\ref{ee1})--(\ref{ee3}) for $k = 1, 2, \ldots, N$ and (\ref{ee4}) for
$k = 1$.  As $N$ may be taken arbitrarily large, inequalities
(\ref{ee1})--(\ref{ee3}) in fact hold for all $k \geq 1$.  Convergence to $\phi(\infty)$
as $s \to \infty$ in the $C_x^\infty(\rr \to \cM)$ topology is an easy consequence
of the energy estimates (\ref{ee3}).

It remains to establish (\ref{ee4}) for $k > 1$.  These follow as a corollary
from (\ref{ee1}) and are shown in the following proposition.
\end{proof}

\begin{proposition}
Let $\phi$ be a heat flow with classical initial data.  Then for sufficiently small
initial energy $E_0$, we have
\begin{equation}
\int_0^\infty s^{k-1} \lVert \e_k(s) \rVert_{L_x^\infty(\rr)} ds \lesssim_k E_0,
\label{IntReg}
\end{equation}
for $k \geq 1$.
\label{Ek Integrated}
\end{proposition}
\begin{proof}

We have already established the $k = 1$ case.  For $k > 1$,
the proof procedes as in \cite{T4} and for convenience we reproduce the
short argument here.

The Gagliardo-Nirenberg inequality (\ref{GN4}) implies
\begin{equation*}
\lVert \sqrt{\e_k}(s) \rVert_{L_x^\infty}
\lesssim
\lVert \sqrt{\e_k}(s) \rVert_{L_x^2}^{1/3}
\lVert \partial_x \sqrt{\e_k}(s) \rVert_{L_x^4}^{2/3}
\end{equation*}
and so in view of the diamagnetic inequality (\ref{diamagnetic})
it holds that
\begin{equation*}
\lVert \sqrt{\e_k}(s) \rVert_{L_x^\infty}
\lesssim
\lVert \sqrt{\e_k}(s) \rVert_{L_x^2}^{1/3}
\lVert \sqrt{\e_{k+1}}(s) \rVert_{L_x^4}^{2/3}.
\end{equation*}
From the Gagliardo-Nirenberg inequality (\ref{GN5}) we have
\begin{equation*}
\lVert \sqrt{\e_{k+1}}(s) \rVert_{L_x^4}
\lesssim
\lVert \sqrt{\e_{k+1}}(s) \rVert_{L_x^2}^{1/2}
\lVert \partial_x \sqrt{\e_{k+1}}(s) \rVert_{L_x^2}^{1/2}.
\end{equation*}
Combining the above inequalities yields
\begin{equation}
\lVert \sqrt{\e_k}(s) \rVert_{L_x^\infty}
\lesssim
\lVert \sqrt{\e_{k}}(s) \rVert_{L_x^2}^{1/3}
\lVert \sqrt{\e_{k+1}}(s) \rVert_{L_x^2}^{1/3}
\lVert \sqrt{\e_{k+2}}(s) \rVert_{L_x^2}^{1/3}.
\label{Inteprep}
\end{equation}
Using (\ref{Inteprep}) in (\ref{IntReg}), we conclude with
H\"older's inequality and (\ref{ee1}) and (\ref{ee2}).
\end{proof}

\subsection{A controlling quantity}
Define over any time interval $I$ the following ``norm'':
\begin{equation*}
\lVert \phi \rVert_{\Sn(I)} := \lVert \lvert \partial_x \phi \rvert_{\phi^*h} \rVert_{L_{s,x}^4(I \times \rr)}.
\end{equation*}
In terms of energy densities, we may equivalently write
\begin{align*}
\lVert \phi \rVert_{\Sn(I)} &= \lVert \sqrt{\e_1} \rVert_{L_{s,x}^4(I \times \rr)} \\
&= \lVert \e_1 \rVert_{L_{s,x}^2(I \times \rr)}^{\frac{1}{2}}.
\end{align*}

Strictly speaking, $\Sn$ is not a norm or even a seminorm on smooth maps
$I \times \rr \to \cM$, but can be interpreted as an intrinsic nonlinear analogue of the
$\dot{W}_x^{1, 4}$ seminorm.
Abusing terminology, we refer to $\Sn$ as the $\Sn$-norm.
Note that the $\Sn$-norm is critical, as it is invariant with respect to the scaling (\ref{dil sym}).  
It is also clear that the $\Sn$-norm is monotonic in the sense that 
it is non-decreasing in time.

In the following let $\phi$ be a partial heat flow
whose initial data $\phi(0) = \phi_0$ has
initial energy $E_0 := E(\phi_0)$.
In this subsection we shall show that the various energy quantities
comprising the left-hand-sides of (\ref{ee1})--(\ref{ee4})
are all controlled by the $\Sn$-norm of $\phi$.
Throughout we use
\begin{equation*}
\Sn_{[0, r)} := \lVert \phi \rVert_{\Sn([0, r))}
\end{equation*}
as a notational shorthand.
Note that in this subsection we do not require $E_0 < \Ec$: the results are valid
so long as the $\Sn$-norm is finite.
However, in applications in this paper we shall take $E_0 < \Ec$ and
in $\S \ref{S:Minimal blowup solutions}$ we prove finiteness
of the $\Sn$-norm under this assumption.
For simplicity the following propositions are stated assuming the initial data
is given at initial time $s = 0$, though the proofs are valid for any
initial time $s = s_0 \geq 0$ we might choose.

\begin{proposition}
Let $\phi$ be a heat flow on $[0, r)$ with initial energy $E_0$.
Then
\begin{equation*}
\int_0^{r} \bE_2(s) \ds \lesssim_{E_0, \Sn_{[0, r)}}  1.
\end{equation*}
\label{e01s1}
\end{proposition}
\begin{proof}
Using (\ref{intEineq}) with $k = 1$, we have
\begin{equation*}
2\int_0^{r} \bE_2(s)\ds + \bE_1(r) - 2E_0 \lesssim \lVert \phi \rVert_{\Sn([0, r))}^4.
\end{equation*}
By monotonicity of the energy $\bE_1/2$ the claim follows.
\end{proof}

\begin{proposition}
Let $\phi$ be a heat flow on $[0, r)$ with initial energy $E_0$.  
Then
\begin{equation*}
\lVert \e_1 \rVert_{L_s^1 L_x^\infty([0, r) \times \rr)} 
\lesssim_{E_0, \Sn_{[0, r)}} 1.
\end{equation*}
\label{e01s2}
\end{proposition}
\begin{proof}
Due to Duhamel (\ref{Duhamel inequality})
we have for all $s > r_0 \geq 0$ that
\begin{equation}
\lVert \e_1(s) \rVert_{L_x^\infty} 
\lesssim
\lVert e^{s\Delta} \sqrt{\e_1}(r_0)\rVert_{L_x^\infty}^2 +
\left\lVert \int_{r_0}^s e^{(s - s^\prime)\Delta} \sqrt{\e_1}^3(s^\prime, x) ds^\prime \right\rVert_{L_x^\infty}^2.
\label{sqdu}
\end{equation}
We upper-bound the nonlinear term using
\begin{equation*}
\left\lVert \int_{r_0}^s e^{(s - s^\prime)\Delta} 
\sqrt{\e_1}^3(s^\prime, x) ds^\prime \right\rVert_{L_x^\infty}^2
\leq
\left( \int_{r_0}^s \lVert e^{(s - s^\prime)\Delta} \sqrt{\e_1}^3(s^\prime, x) \rVert_{L_x^\infty} ds^\prime \right)^2.
\end{equation*}
Integrating in $s$ from $r_0 \geq 0$ to $t \leq r$, we have
by Minkowski's inequality that
\begin{align}
\int_{r_0}^{t} 
&\left( \int_{r_0}^s \lVert e^{(s - s^\prime)\Delta} \sqrt{\e_1}^3(s^\prime, x) \rVert_{L_x^\infty} ds^\prime \right)^2 ds \nonumber \\
&\leq
\left( \int_{r_0}^{t} 
\left( \int_{s^\prime}^{t}
\lVert e^{(s - s^\prime) \Delta} \sqrt{\e_1}^3 (s^\prime, x) \rVert_{L_x^\infty}^2 ds 
\right)^{1/2} ds^\prime
\right)^2.
\label{MI}
\end{align}
From the 
(\ref{Strichartz}) with $p = \infty$ it follows that
\begin{equation*}
 \int_{s^\prime}^{t}
\lVert e^{(s - s^\prime) \Delta} \sqrt{\e_1}^3 (s^\prime, x) \rVert_{L_x^\infty}^2 ds 
\lesssim
\lVert \sqrt{\e_1}^3(s^\prime) \rVert_{L_x^2}^2,
\end{equation*}
and so the right hand side of (\ref{MI}) is controlled by
\begin{align*}
\left( \int_{r_0}^{t} \lVert \sqrt{\e_1}^3(s) \rVert_{L_x^2} ds \right)^2
&\lesssim
\left( \int_{r_0}^{t} \lVert \sqrt{\e_1}(s) \rVert_{L_x^\infty} \lVert \e_1(s) \rVert_{L_x^2} ds \right)^2 \\
&\lesssim
\left( \left( \int_{r_0}^{t} \lVert \sqrt{\e_1}(s) \rVert_{L_x^\infty}^2 ds \right)^{1/2}
\left( \int_{r_0}^{t} \lVert \e_1(s) \rVert_{L_x^2}^2 ds \right)^{1/2} \right)^2 \\
&= \int_{r_0}^{t} \lVert \e_1(s) \rVert_{L_x^\infty} \ds \cdot
\int_{r_0}^{t} \lVert \e_1(s) \rVert_{L_x^2}^2 ds.
\end{align*}

By using (\ref{Strichartz}) with $p = \infty$
to bound the linear term appearing in (\ref{sqdu}), we conclude
\begin{equation*}
\int_{r_0}^{t} \lVert \e_1(s) \rVert_{L_x^\infty} ds \lesssim E_0
+
\int_{r_0}^{t} \lVert \e_1(s) \rVert_{L_x^\infty} \ds \cdot
\int_{r_0}^{t} \lVert \e_1(s) \rVert_{L_x^2}^2 ds.
\end{equation*}
The proposition now follows from splitting up the time interval $[0, r)$
into $O((\Sn_{[0, r)})^2)$ intervals $I_j$
on which
$\lVert \e_1 \rVert_{L_s^2 L_x^2 (I_j \times \rr)}$ is sufficiently small.
\end{proof}

\begin{proposition}
Let $\phi$ be a heat flow on $[0, r)$ with initial energy $E_0$.
Then
\begin{equation*}
\sup_{ \substack{0 < s < r \\ x\in \rr}} s \e_1(s, x) \lesssim E_0
\exp(2 \lVert \e_1 \rVert_{L_s^1L_x^\infty([0, r) \times \rr)}).
\end{equation*}
\end{proposition}
\begin{proof}
Duhamel (\ref{Duhamel inequality}) and parabolic regularity (\ref{Young}) imply
\begin{equation*}
\lVert \sqrt{\e_1}(r) \rVert_{L_x^\infty} \lesssim r^{-1/2}E_0^{1/2} + 
\int_0^r \lVert e^{(r - s)\Delta} \sqrt{\e_1}^3(s) \rVert_{L_x^\infty} \ds.
\end{equation*}
By Young's inequality,
\begin{equation*}
\lVert \sqrt{\e_1}(r) \rVert_{L_x^\infty} \lesssim r^{-1/2}E_0^{1/2} + 
\int_0^r \lVert \sqrt{\e_1}(s) \rVert_{L_x^\infty} 
\cdot \lVert \e_1(s) \rVert_{L_x^\infty} \ds,
\end{equation*}
and so by Gronwall's inequality,
\begin{equation*}
\lVert \sqrt{\e_1}(r) \rVert_{L_x^\infty} \lesssim r^{-1/2} E_0^{1/2} \exp\left(
\int_0^r \lVert \e_1(s)  \rVert_{L_x^\infty} \ds \right).
\end{equation*}
\end{proof}

\begin{lemma}
Let $\phi: I \times \rr \to \cM$ be a heat flow with 
classical initial data $\phi_0$ 
with energy $E_0 := E(\phi_0)$ and
suppose that $\lVert \phi \rVert_{\Sn(I)} < \infty$.  
Setting $S_I := \lVert \phi \rVert_{\Sn(I)}$ for short, we
have for each $k \geq 1$ that
\begin{align}
\int_I \int_\rr s^{k-1} \e_{k+1}(s, x) \dx ds 
&\lesssim_{E_0, k, S_I} 1, \label{e0ks1} \\
\sup_{s \in I} s^{k-1} \int_\rr \e_k(s, x) \dx 
&\lesssim_{E_0, k, S_I} 1, \label{e0ks2} \\
\sup_{\substack{ s \in I \\ x \in \rr}} s^k \e_k(s, x) 
&\lesssim_{E_0, k, S_I} 1, \label{e0ks3} \\
\int_I s^{k-1} \sup_{x \in \rr} \e_k(s, x) \ds 
&\lesssim_{E_0, k, S_I} 1 \label{e0ks4}.
\end{align}
\label{Preliminary Estimates}
\end{lemma}
\begin{proof}
The proof is by induction.
The base case ($k = 1$) of (\ref{e0ks1}) follows from Proposition \ref{e01s1},
that of (\ref{e0ks2}) from Proposition \ref{EnergyMono},
that of (\ref{e0ks3}) from the argument of Proposition \ref{B to the C}, 
and finally that of (\ref{e0ks4}) from Proposition \ref{e01s2}. 

Now assume $k > 1$ and that the inequalities of the lemma have been proven
for all smaller $k$.  Then we have at our disposal (\ref{AGM}) 
as in \S \ref{S:Small Energy Estimates}.
Repeating the argument of Proposition \ref{A1 Implies Ak} implies
(\ref{e0ks1}).  The proof of (\ref{e0ks2}) follows from the argument used in 
Proposition \ref{Aks Imply Bk}.  Inequalities (\ref{e0ks3}) follow from
Proposition \ref{B to the C}.  Finally, we conclude (\ref{e0ks4}) using
the argument of Proposition \ref{Ek Integrated}.
\end{proof}

\section{Local energy estimates} \label{S:Energy Estimates}
To prepare for the minimal blow-up solution argument of the next section
we now establish some local energy estimates.
In view of this and for the sake of convenience we introduce the new notation
\begin{equation*}
E(\phi(s) \upharpoonleft U) := \frac{1}{2} \int_U \e_1(s) \dx,
\end{equation*}
and similarly,
\begin{equation*}
\bE_k(\phi(s) \upharpoonleft U) := \int_U \e_k(s) \dx,
\end{equation*}
where $U \subset \rr$.  
Usually we take $U = B(x, R)$, that is, equal to an open ball 
in $\rr$ centered at the point $x$ and of radius $R$.  
The following two lemmas are intrinsic formulations of
\cite[Lemmas 3.1, 3.2]{St85}.
\begin{lemma}(Energy localization).
Let $\phi$ be a heat flow on $[0, S] \times \rr$.  
Then for any $s \in [0, S]$ and $R > 0$ we have
\begin{align}
\lVert \e_1 \rVert^2_{L_{s,x}^2([0, s] \times \rr)} \lesssim&
\; \lVert \bE_1(\phi \upharpoonleft B(\cdot, R)) 
\rVert_{L_{s,x}^\infty([0, s] \times \rr)} \times \nonumber \\
&\times
\left(
\lVert \e_2 \rVert_{L_{s,x}^1 ([0, s] \times \rr)} + \frac{1}{R^2}
\lVert \e_1 \rVert_{L_{s,x}^1 ([0, s] \times \rr)}
\right).
\label{EnergyLocalization}
\end{align}
\label{EL}
\end{lemma}
Lemma \ref{EL} follows as a corollary from the following result
via a simple covering argument as 
in \cite{St85}.

\begin{lemma}
Let $\phi$ be a heat flow on $[0, S] \times \rr$.  
Let $x \in \rr$, $R > 0$, set $B := B(x, R)$, and let $\chi_B$ be the indicator function
of $B$.
Then for any $s \in [0, S]$ and $R > 0$ we have
\begin{align*}
\lVert \e_1 \cdot \chi_B \rVert^2_{L_{s,x}^2([0, s] \times \rr)} \lesssim&
\; \lVert \bE_1(\phi \upharpoonleft B(x, R)) \rVert_{L_{s}^\infty([0, s])} \times \nonumber \\
&\times
\left(
\lVert \e_2 \cdot \chi_B \rVert_{L_{s,x}^1 ([0, s] \times \rr)} + \frac{1}{R^2} 
\lVert \e_1 \cdot \chi_B \rVert_{L_{s,x}^1 ([0, s] \times \rr)}
\right).
\end{align*}
\end{lemma}
\begin{proof}
Letting
\begin{equation*}
(\sqrt{\e_1})_\mathrm{ ave} := \frac{1}{\pi R^2} \int_B \sqrt{\e_1} \dx,
\end{equation*}
we have by Poincar\'e's inequality that
\begin{align}
\int_0^{s} \int_{B} \e_1^2 \dx ds^\prime &\lesssim
\int_0^{s} \int_{B} \lvert \sqrt{\e_1} - (\sqrt{\e_1})_\mathrm{ ave} \rvert^4 \dx ds^\prime +
\int_0^{s} \int_{B} (\sqrt{\e_1})^4_\mathrm{ ave} \dx ds^\prime \nonumber \\
&\lesssim 
\sup_{0 < s^\prime < s} \int_{B} \lvert \sqrt{\e_1} - (\sqrt{\e_1})_\mathrm{ ave} \rvert^2 \dx
\int_0^{s} \int_{B} \lvert \partial_x \sqrt{\e_1} \rvert^2 \dx ds^\prime \nonumber\\
&\quad + \int_0^{s} (\pi R^2)^{-3} \left\lvert \int_{B} \sqrt{\e_1} \dx \right\rvert^4 ds^\prime 
\label{Poin}.
\end{align}
Now
\begin{equation}
\int_B \lvert \sqrt{\e_1} - (\sqrt{\e_1})_\mathrm{ ave}\rvert^2 \dx \leq
\int_B \e_1 \dx.
\label{minimizer}
\end{equation}
From H\"older's inequality it follows that
\begin{align}
\left\lvert \int_B \sqrt{\e_1}(s^\prime) \dx \right\rvert^4
&\leq 
(\pi R^2)^2 \left( \int_B \e_1 \dx \right)^2 \nonumber \\
&\leq 
(\pi R^2)^2 \left\lVert \int_B \e_1(\cdot) \dx \right\rVert_{L_s^\infty([0, s^\prime])} 
\int_B \e_1(s^\prime) \dx.
\label{part2}
\end{align}
Substituting (\ref{minimizer}) and (\ref{part2}) in (\ref{Poin}) and using the diamagnetic inequality (\ref{diamagnetic}) in (\ref{Poin}) proves the lemma.
\end{proof}

\begin{lemma}(Energy concentration).
There exists a constant $C = C(\rr, \cM)$ such that for any $S > 0$ and any
heat flow $\phi$ on $[0, S] \times \rr$ we have for any 
$0 \leq s \leq S$, $x \in \rr$, and $R > 0$ that
\begin{equation}
E(\phi(s) \upharpoonleft B(x, R)) \leq
E(\phi(0) \upharpoonleft B(x, 2R)) +
C \frac{s}{R^2}E_0.
\label{EnergyConcentration}
\end{equation}
\end{lemma}
\begin{proof}
Let $\rho \in C_c^\infty(\rr \to [0, 1])$ be compactly supported in $B_{2R}(x)$, equal to $1$ on $B_R(x)$, and such that
$\lvert \partial_x \rho \rvert \leq 2/R$.
Then via (\ref{heatflowequation}), integration by parts, (\ref{NoTorsion2}), 
and (\ref{Leibniz2}), we obtain
\begin{align*}
\int_\rr \langle \partial_s \phi, \partial_s \phi \rangle_\ph \rho^2 \dx
=& \; \int_\rr \langle \rho^2 \partial_s \phi, (\pd)_j \partial_j \phi \rangle_\ph \dx \\
=&\; - \int_\rr \langle \rho^2 (\pd)_j \partial_s \phi, \partial_j \phi \rangle_\ph \rho^2 \dx \\
&- 2 \int_\rr \langle \rho \partial_j \rho \partial_s \phi, \partial_j \phi \rangle_\ph \dx \\
=&\; - \int_\rr \langle (\pd)_s \partial_j \phi, \partial_j \phi \rangle_\ph \rho^2 \dx \\
&- 2 \int_\rr \langle \rho \partial_j \rho \partial_s \phi, \partial_j \phi \rangle_\ph \dx \\
=& \; - \partial_s \int_\rr \e_1 \rho^2 \dx
- 2 \int_\rr \langle \partial_s \phi, \partial_j \phi \rangle_\ph \rho \cdot \partial_j \rho \dx.
\end{align*}
By Cauchy-Schwarz we have
\begin{equation*}
\int_\rr \langle \partial_s \phi, \partial_s \phi \rangle_\ph \rho^2 \dx
+ \partial_s \int_\rr \e_1 \rho^2 \dx 
\lesssim
\int_\rr \lvert \partial_s \phi \rvert \lvert \partial_x \phi \rvert_\ph \lvert \partial_x \rho \rvert_\ph \rho \dx,
\end{equation*}
and so from Young's inequality it follows that
\begin{equation*}
\int_\rr \langle \partial_s \phi, \partial_s \phi \rangle_\ph \rho^2 \dx
+ \partial_s \int_\rr \e_1 \rho^2 \dx 
\leq
\int_\rr \lvert \partial_s \phi \rvert^2_\ph \rho^2 \dx +
\frac{C}{R^2} \int_\rr \lvert \partial_x \phi \rvert^2_\ph \dx.
\end{equation*}
Therefore, in view of the smoothness of $\phi$ and the fact that $\rho$ has compact support,
it holds that
\begin{equation}
\partial_s \int_\rr \e_1 \rho^2 \dx \lesssim R^{-2}  \int_\rr \e_1 \dx.
\label{EL ODE}
\end{equation}
Integrating in (\ref{EL ODE}), we conclude
\begin{align*}
E(\phi(s) \upharpoonleft B(x, R)) 
&\leq \frac{1}{2} \int_\rr \e_1(s) \rho^2 \dx \\
&\leq \frac{1}{2} \int_\rr \e_1(0) \rho^2 \dx + \frac{C}{R^2} \int_0^{s} \int_\rr \e_1(s^\prime) \dx ds^\prime \\
& \leq E(\phi(0) \upharpoonleft B(x, 2R)) + C \frac{s}{R^2} E_0.
\end{align*}
\end{proof}

\section{Minimal blowup solutions} \label{S:Minimal blowup solutions}
Recall that the ground state energy $\Ec$ is defined to be equal to the minimum energy carried by a 
finite energy
nontrivial harmonic map $\rr \to \cM$ provided such maps exist and is set equal to $+\infty$ 
if such maps do not exist.
Our goal is to show that there exists an energy-dependent constant $M$
such that
\begin{equation}
\lVert \phi \rVert_{\Sn([0, \infty))} \leq M(E) < \infty
\label{Sbound}
\end{equation}
whenever $E(\phi(0)) \leq E < \Ec$.
It follows from interpolating between the small energy bounds
(\ref{ee2}) (with $k = 1$) and (\ref{ee4}) (with $k = 1$)
of Theorem \ref{SmallEnergyResult} that
such an inequality holds for $E$ sufficiently small.
Let $E^*$ be the supremum of $\tilde{E}$ 
for which $\lVert \phi \rVert_{\Sn[0, \infty)} \leq M(\tilde{E}) < \infty$ 
whenever $E(\phi(0)) \leq \tilde{E}$.

\begin{proposition}
It holds that
\begin{equation*}
E^* = \Ec.
\end{equation*}
\end{proposition}
\begin{proof}
The proof will proceed by contradiction.  It is clear
that necessarily $E^* \leq \Ec$, because nontrivial harmonic maps
as stationary solutions cannot satisfy (\ref{Sbound}).
Hence we shall assume $E^* < \Ec$.  By taking a family
of initial data with energy strictly between $E^*$ and $\Ec$
and tending to $E^*$, we may extract along some subsequence
and in a suitably strong sense a nontrivial harmonic map.
Crucial use is made of the definition of $E^*$ in order to establish this.
However, as energy is nonincreasing along the flow, the heat evolution $\phi$
always has energy less than and bounded away from the minimum energy $\Ec$ 
required to form a nontrivial harmonic map, which provides the contradiction.
Nowhere must we assume that $\Ec$ is finite, and thus our argument is 
also valid in the case where $\cM$ does not admit any finite energy harmonic maps.

Let $\eta > 0$ be a small parameter ($\eta \ll 1, E^*, \Ec$).  For any
$\phi$ with classical initial data $\phi(0)$ having energy $E(\phi(0)) < \Ec$,
let $s^\prime = s^\prime(\phi)$ be the first time such that
\begin{equation*}
\lVert \phi \rVert_{\Sn([0, s^\prime])} = \eta.
\end{equation*}
If there is no such time $s^\prime$ or if $s^\prime = \infty$, 
then Lemma \ref{Preliminary Estimates} is in force globally;
therefore we may assume without loss of generality that the $s^\prime$ 
that we consider are finite.

Let $\epsilon > 0$ be such that $\epsilon \ll \eta$.  Suppose $E^* < \Ec$.
If $E$ satisfies $E^* < E < \Ec$ and yet $E(\phi(s^\prime)) < E^* - \epsilon$
independently of $\phi$, then there exists a bound of the form (\ref{Sbound}) for $E$,
thus contradicting the choice of $E^*$.  
Therefore for all $n$ large there exists a heat flow $\phi_n$ with energy 
$E^* < E(\phi_n(0)) < E^* + 1/n < \Ec$ such that
\begin{equation*}
E(\phi_n(s^\prime_n)) > E^* - \frac{1}{n}.
\end{equation*}
Using scale invariance of the energy with respect to (\ref{dil sym})
 we may assume that each $s^\prime_n = 1$.

By (\ref{EnergyLocalization}) it follows that for all $R > 0$
\begin{align}
\eta^4 = \lVert \phi_n \rVert_{\Sn[0, 1]}^4 \lesssim& \;
\lVert E(\phi_n \upharpoonleft B_R(x) ) \rVert_{L^\infty_{s, x}([0, 1] \times \rr)}
\times \nonumber\\
&\times
\left(
\lVert \sqrt{\e_2} \rVert_{L_{s,x}^2([0, 1] \times \rr)}^2
+ \frac{1}{R^2} \lVert \sqrt{\e_1} \rVert^2_{L_{s,x}^2([0, 1] \times \rr)}
\right).
\label{Struwe1}
\end{align}
Taking $k = 1$ in (\ref{intEineq}), we have
\begin{equation}
\bE_1(1) - \bE_1(0) 
+ 2 \lVert \sqrt{\e_2} \rVert_{L_{s,x}^2 ([0, 1] \times \rr)}^2
\lesssim
\lVert \sqrt{\e_1} \rVert_{L_{s,x}^4([0, 1] \times \rr)}^4.
\label{BW}
\end{equation}
For large $n$, $E(\phi_n(1)) \sim E^*$ by construction and therefore
from (\ref{BW}) it follows that $\lVert \sqrt{\e_2} \rVert^2_{L_s^2L_x^2([0, 1]\times \rr)} \lesssim_{E^*} \eta^4$.  
Using both of these facts in (\ref{Struwe1}) (and suppressing now the dependence upon $E^*$ 
in subsequent inequalities) we get for $R > 0$ that
\begin{equation*}
\eta^4 \lesssim \lVert E(\phi_n \upharpoonleft B_R(x) ) \rVert_{L^\infty_{s, x}([0, 1] \times \rr)}
\left( \eta^4 + \frac{1}{R^2} \right).
\end{equation*}
For large $n$ pick $r_n > 0$ so that 
$r_n \sim_\eta 1$
and
\begin{equation*}
1\lesssim  \lVert E(\phi_n \upharpoonleft B_{r_n}(x) ) \rVert_{L^\infty_{s, x}([0, 1] \times \rr)}.
\end{equation*}
Thus for all large $n$ we may pick $x_n \in \rr$ and $0 \leq s_n \leq 1$ so that
\begin{equation*}
1 \lesssim_\eta E(\phi(s_n) \upharpoonleft B_{r_n}(x_n)).
\end{equation*}
By (\ref{EnergyConcentration}) we have
\begin{equation*}
E(\phi_n(s_n) \upharpoonleft B_{r_n}(x_n)) \leq
E(\phi_n(s) \upharpoonleft B_{2r_n}(x_n)) + C \frac{s_n - s}{r_n^2} E(\phi_n(s)),
\label{Struwe2}
\end{equation*}
for $0 \leq s \leq s_n$,
which in view of our choice of $x_n$, $r_n$, and $s_n$, means
\begin{equation*}
1 \lesssim_\eta 
E(\phi_n(s) \upharpoonleft B_{2r_n}(x_n)) + C \frac{s_n - s}{r_n^2} E(\phi_n(s)).
\end{equation*}
As $r_n \sim 1$ and $0 \leq s_n - s \leq 1$, 
we may choose $\tilde{r}_n \geq r_n$, $\tilde{r}_n \sim_{\eta} r_n$
so that
\begin{equation*}
1 \lesssim_\eta E(\phi_n(s) \upharpoonleft B_{\tilde{r}_n}(x_n))
\end{equation*}
for all $0 \leq s \leq s_n$.
Set
\begin{equation*}
\tilde{\phi}_n(s, x) := \phi_n(s \tilde{r}_n^2, x_n + x \tilde{r}_n).
\end{equation*}
Since
$s_n /\tilde{r}_n^2 \gtrsim  1$,
we have
for some sufficiently small constant $c > 0$
that each $\tilde{\phi}_n$ is a heat flow with 
\begin{equation*}
1 \lesssim_\eta E(\tilde{\phi}_n(s) \upharpoonleft B(0,1))
\end{equation*}
for all $0 \leq s \leq c \eta$.
Set $I = [0, c \eta]$.
To each $\tilde{\phi}_n$ we have associated its tension field $\T_n$ given by
\begin{equation*}
\T_n := (\tilde{\phi}_n^* \nabla)_x \partial_x \tilde{\phi}_n.
\end{equation*}
Then
\begin{equation*}
\lVert \T_n \rVert_{L_{s,x}^2(I \times \rr)} \to 0 \quad \mathrm{ as}\quad n \to \infty
\end{equation*}
since $E(\phi_n(1)) - E(\phi_n(0)) \to_{n \to \infty} 0$ 
(by construction)
and because
the energy functional is monotonic (Proposition \ref{EnergyMono}).
Passing to a subsequence if necessary, choose times $\tilde{s}_n \in I$ so that
\begin{equation*}
\lVert \T_n(\tilde{s}_n) \rVert_{L^2_x(\rr)} \to 0.
\end{equation*}
Set $\psi_n(s) := \tilde{\phi}_n(s - \tilde{s}_n)$.  
Therefore the tension fields associated to $\psi_n$ satisfy
\begin{equation*}
\lVert (\psi_n^* \nabla)_x \partial_x \psi_n (0) \rVert_{L^2_x(\rr)} \to 0,
\end{equation*}
and so in view of \cite[Proposition 5.1]{St85}
the $\psi_n(0)$ converge strongly 
along some subsequence
in the Sobolev space $H_{x, \mathrm{ loc}}^{2}$ 
to a harmonic map $\psi_0$ 
away from any ``energy bubbles'', which asymptotically have
a harmonic map profile.  
As $E(\psi_n(0)) < \Ec - \epsilon^\prime$ 
uniformly in $n$ for some $\epsilon^\prime > 0$, it follows that energy bubble formation
is impossible and that $\psi_0$ must be a 
trivial harmonic map.  This however is precluded by the fact that
$E(\psi_n(0) \upharpoonleft B(0,1)) \gtrsim_\eta 1$ uniformly in $n$.  
From this contradiction we conclude that $E^* = \Ec$.
\end{proof}

Combining Lemma \ref{Preliminary Estimates} and Theorem \ref{LWP}, 
we conclude from $E^* = \Ec$ the following
\begin{thm}
For any classical initial data $\phi_0$ with $E(\phi_0) < \Ec$ we have that there exists a unique global smooth heat flow $\phi$ with initial data $\phi_0$.  Moreover, $\phi$ satisfies the estimates
\begin{align}
\int_0^\infty \int_\rr s^{k-1} \e_{k+1}(s, x) \dx ds &\lesssim_{E_0, k} 1, \label{ek1} \\
\sup_{0 < s < \infty} s^{k-1} \int_\rr \e_k(s, x) \dx &\lesssim_{E_0, k} 1, \label{ek2} \\
\sup_{ \substack{ 0 < s < \infty \\ x \in \rr}} s^k \e_k(s, x) &\lesssim_{E_0, k} 1, \label{ek3} \\
\int_0^\infty s^{k-1} \sup_{x \in \rr} \e_k(s, x) \ds &\lesssim_{E_0, k} 1. \label{ek4}
\end{align}
for each $k \geq 1$.
\label{Main Energy Estimates}
\end{thm}

\begin{corollary}
Let $\phi_0$, $\phi$ be as in Theorem \ref{Main Energy Estimates},
classical with respect to $\phi(\infty) \in \cM$.  Then
for each fixed $s$ the function $\phi(s)$ is Schwartz with respect to $\phi(\infty)$.
Additionally, $\phi(s)$ converges as $s \to \infty$ in the 
$C_x^\infty(\rr \to \mathbf{R}^+ \times \cM)$ topology
to $\phi(\infty)$.
\label{Convergence Cor}
\end{corollary}
\begin{proof}
By Theorems \ref{LWP} and \ref{Main Energy Estimates} and a simple continuity argument
we have that $\phi(s) - \phi(\infty)$ is Schwartz for each fixed $s \geq 0$.

Energy bounds imply $\phi(s, x) \to_{s \to \infty} \phi(\infty)$ in the uniform topology and that
$\partial_x^k \phi(s, x) \to_{s \to \infty} 0$ in the uniform topology for all integers $k \geq 1$.
Therefore, $\phi(s)$ converges to $\phi(\infty)$ in $C_x^\infty(\rr \to \cM)$
as $s \to \infty$.  
To upgrade this to convergence in 
$C_x^\infty(\rr \to \mathbf{R}^+ \times \cM)$,
use (\ref{heatflow ex}) to convert time derivatives into spatial derivatives.
\end{proof}

\section{The caloric gauge} \label{S:Caloric gauge}

For simplicity, we split this section into two parts, in the first ignoring the time variable $t$,
and in the second taking it up again.

\subsection{The static caloric gauge} \label{S:StaticCaloricGauge}
Let $\phi: \mathbf{R}^+ \times \rr \to \cM$ be a heat flow with classical initial data $\phi_0$
that equals $\phi(\infty)$ at spatial infinity.  We now introduce the notion
of a gauge, followed by the \emph{caloric gauge} which originates in \cite{Trenorm}.  
Our conventions and presentation draw on \cite{Trenorm}, \cite{T4}, and
\cite[Chapter 6]{Tdis}.  For a thorough introduction to gauges (sometimes referred
to as \emph{moving frames}), see \cite[Chapter 7]{Spivak}.
 
Given any $p \in \cM$ we define an \emph{orthonormal frame} at $p$ to be any orthogonal
orientation-preserving map $e : \mathbf{R}^m \to T_p \cM$ from $\mathbf{R}^m$ to the
tangent space $T_p \cM$ at $p$.  We let $\mathrm{ Frame}(T_p \cM)$ stand for the space of all
such frames and note that this space admits the obvious transitive action of the special orthogonal
group $SO(m)$.
We define the \emph{orthonormal frame bundle} $\mathrm{ Frame}(\phi^*T \cM)$ of $\phi$ to
be the space of all pairs $((s, x), e)$ where $(s, x) \in \mathbf{R}^+ \times \rr$ and
$e \in \mathrm{ Frame}(T_{\phi(s, x)}\cM)$.  This is a fiber bundle over $\mathbf{R}^+ \times \rr$.
We define an \emph{orthonormal frame} $e \in \Gamma(\mathrm{ Frame}(\phi^* T\cM))$ for $\phi$
to be a section of $\mathrm{ Frame}(\phi^* T\cM)$, i.e., a smooth assignment
$e(s, x) \in \mathrm{ Frame}(T_{\phi(s, x)}\cM)$ of an orthonormal frame at $\phi(s, x)$ to
every point $(s, x) \in \mathbf{R}^+ \times \rr$.

Each orthonormal frame $e \in \Gamma(\mathrm{ Frame}(\phi^* T\cM))$ provides
an orthogonal, orientation-preserving identification between the vector bundle $\phi^*T\cM$
bearing the metric $\ph$ and the trivial bundle $(\mathbf{R}^+ \times \rr) \times \mathbf{R}^m$
bearing the Euclidean metric on $\mathbf{R}^m$.
Hence via $e$ we may pull back sections $\Psi \in \Gamma (\phi^* T \cM)$ to functions $e^* \Psi$ and 
we may pull back
the connection $\pd$ on $\phi^* T\cM$ to a connection $D$ on the trivial bundle $(\mathbf{R}^+ \times \rr) \times \mathbf{R}^m$ so that
\begin{equation}
D_\alpha := \partial_\alpha + A_\alpha
\label{D Def}
\end{equation}
with $A_\alpha \in \mathfrak{so}(m)$ the skew-adjoint $m \times m$ matrix field defined by
\begin{equation}
(A_\alpha)_{ab} := \langle (\pd)_\alpha e_a, e_b\rangle_{\ph}
\label{A Definition}
\end{equation}
and $e_1, \ldots, e_m$ the images under $e$ of the standard orthonormal basis vectors of $\mathbf{R}^m$.

Throughout this section Greek lettered indices denote arbitrary variables, i.e. space or time, whereas
Roman lettered indices are only used to denote spatial variables.

Let us define the derivative fields $\psi_\alpha : \mathbf{R}^+ \times \rr \to \mathbf{R}^m$ by
\begin{equation}
\psi_\alpha := e^* \partial_\alpha \phi.
\label{DF Def}
\end{equation}
For short we write $\psi_x = (\psi_1, \psi_2)$ and $A_x = (A_1, A_2)$.
The zero-torsion property (\ref{NoTorsion2}) now manifests itself as
\begin{equation}
D_\alpha \psi_\beta = D_\beta \psi_\alpha,
\label{NoTorsion3}
\end{equation}
which equivalently may be expressed as
\begin{equation*}
\partial_\alpha \psi_\beta - \partial_\beta \psi_\alpha = A_\beta \psi_\alpha - A_\alpha \psi_\beta.
\end{equation*}
The Leibnitz rule (\ref{Leibniz2}) becomes
\begin{equation*}
\partial_\alpha (\psi \cdot \psi^\prime) = 
(D_\alpha \psi) \cdot \psi^\prime +
\psi \cdot (D_\alpha \cdot \psi^\prime).
\end{equation*}
Note that the covariant derivative $D_x$ acts on matrix fields $B$ via
\begin{equation}
D_x B = \partial_x B + [A, B].
\label{Matrix Action}
\end{equation}
Curvature in this context is given by
\begin{equation}
F_{\alpha \beta} := [D_\alpha, D_\beta]
= \partial_\alpha A_\beta - \partial_\beta A_\alpha
+ [A_\alpha, A_\beta].
\label{F Def}
\end{equation}
Using the coordinate expression
\begin{equation*}
(\phi^* \nabla)_\alpha (\phi^* \nabla)_\beta \Psi^c -
(\phi^* \nabla)_\beta (\phi^* \nabla)_\alpha \Psi^c
=
(\partial_\alpha \phi)^a (\partial_\beta \phi)^b
\cR_{a b d}^c \Psi^d,
\end{equation*}
for curvature,
we can write
\begin{equation}
F_{\alpha \beta} = (e^* \cR(\phi))(\psi_\alpha, \psi_\beta),
\label{F R}
\end{equation}
where
$(e^* \cR(\phi))( \cdot, \cdot )$ denotes the pullback of the 
Riemannian curvature tensor $\cR$ on $\cM$,
defined in local coordinates via
\begin{equation*}
(e (e^* \cR(\phi))(\psi_\alpha, \psi_\beta) \phi)^c
=
\cR_{abd}^c(\phi) (e \psi_\alpha)^a (e \psi_\beta)^b (e \phi)^d.
\end{equation*}

In terms of the frame $e$ the heat flow equation (\ref{heatflowequation}) becomes
\begin{equation}
\psi_s = D_j \psi_j.
\label{Frame Heat}
\end{equation}
Note the gauge symmetry inherent in (\ref{Frame Heat}):
It is invariant under the transformation
\begin{align*}
\phi &\mapsto \phi; &e &\mapsto eU;
&\psi_\alpha \mapsto U^{-1} \psi_\alpha; \\
D_\alpha &\mapsto U^{-1}D_\alpha U;&
A_\alpha &\mapsto U^{-1}\partial_\alpha U + U^{-1}A_\alpha U,
\end{align*}
for any choice of \emph{gauge transform} $U: \mathbf{R}^+ \times \rr \to SO(m)$.

We now introduce the caloric gauge, defining it as in \cite{T4}:
\begin{definition}
Let $\phi_0: \rr \to \cM$ be a function 
Schwartz with respect to $\phi(\infty) \in \cM$
and with energy $E_0$ less than $\Ec$.
Let $\phi: \mathbf{R}^+ \times \rr \to \cM$ be its heat flow extension.  We say that a gauge $e$ is a
\emph{caloric gauge} for $\phi$ with
\emph{boundary frame} $e(\infty) \in \mathrm{Frame}(T_{\phi(\infty)}\cM)$
if we have
\begin{equation*}
A_s = 0
\end{equation*}
throughout $\mathbf{R}^+ \times \rr$ and also
\begin{equation*}
\lim_{s \to \infty} e(s, x) = e(\infty)
\end{equation*}
for all $x \in \rr$.
\end{definition}

\begin{thm}[Existence and uniqueness of caloric gauge]
Let $\phi: \mathbf{R}^+ \times \rr \to \cM$
be a heat flow with classical initial data that 
equals $\phi(\infty) \in \cM$ at infinity and that has 
initial energy $E_0$ less than $\Ec$.  Let $e(\infty) \in \mathrm{Frame}(T_{\phi(\infty)}\cM)$ 
be an orthonormal frame at $\phi(\infty)$.  Then there exists a unique caloric gauge 
$e \in \Gamma(\mathrm{Frame}(\phi^* T\cM))$ for $\phi$ with boundary frame $e(\infty)$.
Furthermore, for each fixed $s$, the $\psi_\alpha$ are Schwartz functions in space.
\label{CalGauge EU}
\end{thm}
\begin{proof}
The caloric gauge condition $A_s = 0$ implies that the frame $e$ evolves according to
\begin{equation}
(\phi^* \nabla)_s e_j = 0.
\label{Frame evolution}
\end{equation}
If $e, e^\prime$ are caloric gauges with the same boundary frame $e(\infty),$ then
$\lvert e_j - e_j^\prime \rvert$ is constant in $s$.  Uniqueness follows since
$\lvert e_j - e_j^\prime \rvert$ vanishes as $s \to \infty$.

Turning now to existence, given initial data $\phi_0$,
choose an arbitrary smooth orthonormal frame $e(0, x) \in \mathrm{ Frame}(T_{\phi(0, x)} \cM)$
that differs from a constant frame by a Schwartz function.  
This is possible because we assume that the initial data $\phi_0$ is classical
and because the spatial domain $\rr$ is contractible.
As $\phi^* \nabla$ respects the metric on $\phi^* T\cM$, we have
from (\ref{Frame evolution})
that $e$ remains orthonormal as it evolves in $s$.  
It follows from the Picard theorem that $e$ may be defined globally.
Smoothness of $\phi$ and $e(0, \cdot)$ implies that $e$, $\psi$, and $A$ are all smooth.

Viewing the energy density estimates (\ref{ek3}) in terms of frames yields
\begin{equation}
\lvert D_x^{k-1} \psi_x \rvert
\lesssim_{E_0, k}
s^{-k/2}
\label{Spatial Field Decay}
\end{equation}
for $k \geq 1$.
Using (\ref{Frame Heat}), we also have for $k \geq 1$ that
\begin{equation}
\lvert D_x^{k-1} \psi_s \rvert
\lesssim_{E_0, k}
s^{-(k+1)/2}.
\label{Heat Field Decay}
\end{equation}
Combining $A_s = 0$ and (\ref{F Def}) gives
\begin{equation}
\partial_s A_x = F_{xs}.
\label{A Heat Evolution}
\end{equation}
Therefore, using (\ref{F R}), tensoriality, and bounded geometry, we have
\begin{equation}
\lvert \partial_s D_x^k A_x \rvert
\lesssim_k
\sum_{ \substack{j_1 + j_2 + \cdots j_\ell = k \\ 
j_1, j_2, \ldots, j_\ell \geq 0 \\
\ell = 2, \ldots, k + 2} }
\lvert D_x^{j_1} \psi_s \rvert
\lvert D_x^{j_2} \psi_x \rvert
\cdots
\lvert D_x^{j_\ell} \psi_x \rvert
\label{DA Decay}
\end{equation}
for $k \geq 0$.
Using (\ref{Spatial Field Decay}) and (\ref{Heat Field Decay}) in (\ref{DA Decay})
implies
\begin{equation*}
\lvert \partial_s D_x^k A_x \rvert
\lesssim_k
s^{-(k + 3)/2}.
\end{equation*}
Therefore, since $A_x$ and its derivatives are bounded at $s=0$,
\begin{equation}
\lvert D_x^k A_x \rvert \lesssim_{k, \phi} 1
\label{DA Qualitative}
\end{equation}
on $\mathbf{R}^+ \times \rr$ 
for all $k \geq 0$,
which, combined with (\ref{Matrix Action}) and the triangle inequality, implies
\begin{equation}
\lvert \partial_x^k A_x \rvert
\lesssim_{k, \phi} 1.
\label{dA Qualitative}
\end{equation}
Using (\ref{DA Qualitative}) and (\ref{dA Qualitative}) in
(\ref{Spatial Field Decay}) and (\ref{Heat Field Decay})
leads to
\begin{equation*}
\lvert \partial_x^k \psi_x \rvert
\lesssim_k s^{-1/2}
\quad \text{and} \quad
\lvert \partial_x^k \psi_s \rvert
\lesssim_k s^{-1}.
\end{equation*}
In view of (\ref{A Heat Evolution}),
differentiating $F_{\alpha \beta} = (e^* R(\phi))(\psi_\alpha, \psi_\beta)$ and using
(\ref{DA Qualitative}), (\ref{dA Qualitative}), (\ref{Spatial Field Decay}), (\ref{Heat Field Decay}),
and bounded geometry, we conclude
$A_x(s)$ converges in $C_x^\infty(\rr \to \mathfrak{so}(m))$ to
some limit $A_x(\infty)$ as $s \to \infty$.
Because of this convergence and the fact that $\phi(s)$
is spatially Schwartz uniformly in $s$, it follows from definition (\ref{A Definition})
of the connection coefficients $(A_\alpha)_{ab}$
that $\iota \circ e$, viewed as a linear transformation,
is locally bounded in $C_x^\infty$ uniformly in $s$.

From (\ref{Duhamel}) and (\ref{heatflow ex}) it follows that
\begin{equation*}
\lVert \partial_s \phi(s) \rVert_{L_x^\infty}
\leq
\lVert e^{s \Delta} \phi_0 \rVert_{L_x^\infty}
+
\int_0^s \lVert e^{(s - s^\prime)\Delta}
\left( \Pi(\phi)(\partial_x \phi, \partial_x \phi) \right)(s^\prime)
\rVert_{L_x^\infty} \ds^\prime,
\end{equation*}
wherein $e^{s \Delta}$ retains its meaning given by (\ref{heat kernel}) as is clear
from context (as opposed to referring to the caloric gauge $e$).
From Young's inequality and the fact that $\phi(s)$ is spatially
Schwartz in $x$ uniformly in $s$, we conclude
\begin{equation}
\int_0^\infty \lVert \partial_s \phi(s, \cdot) \rVert_{L_x^\infty} \ds
\lesssim_{\phi} 1,
\label{time deriV bound}
\end{equation}
which, upon integrating in (\ref{Frame evolution}), establishes uniform convergence
of $e(s, \cdot)$ to some limit $e(\infty, \cdot)$.
Using the local $C^\infty$ bounds for $\iota \circ e$
upgrades this to convergence in $C_{x, \mathrm{ loc}}^\infty$,
thereby showing $e(\infty, \cdot)$ is smooth.
Taking limits in the definition (\ref{A Definition}) of $A$
shows that indeed $A_x(\infty, \cdot)$ are the connection coefficients for the frame $e(\infty, \cdot)$.

Applying a smooth gauge transformation 
$U(s, x) = U(x)$ to normalize $e(\infty, \cdot) = e(\infty)$, we get the caloric gauge.
In this gauge $A_x(s)$ converges in $C_\mathrm{ loc}^\infty(\rr \to \mathfrak{so}(m))$
to zero as $s \to \infty$.  In particular, this normalization preserves the caloric gauge
condition $A_s = 0$ and specifies the boundary frame $e(\infty)$.

It remains to show that the $\psi_\alpha$ are Schwartz.  This is an easy
consequence of Corollary \ref{Field Cor} stated below.
\end{proof}

\begin{lemma}(Equations of motion).
Let $\phi: \mathbf{R}^+ \times \rr \to \cM$ be a heat flow with classical initial data 
whose energy $E_0$ is less than $\Ec$, let $e$ be a caloric gauge for $\phi$, 
and let $\psi_x$, $\psi_s$, $A_x$ be the associated derivative fields and connection fields.  
Then we have the evolution equations
\begin{align}
\partial_s \psi_x &= D_x \psi_s = \partial_x \psi_s + A_x \psi_s, \label{EOM1} \\
\partial_s A_x &= F_{xs}, \label{EOM2} \\
\partial_s \psi_\alpha &= D_i D_i \psi_\alpha + F_{\alpha i}. \label{EOM3}
\end{align}
\label{EquationsOfMotion}
\end{lemma}
\begin{proof}
The first equation (\ref{EOM1}) follows from $A_s = 0$, the zero-torsion property (\ref{NoTorsion3}), 
and (\ref{D Def}).

As already noted previously (see (\ref{A Heat Evolution})), 
(\ref{EOM2}) is a consequence of $A_s = 0$ and (\ref{F Def}).

To prove (\ref{EOM3}), we use $A_s = 0$ and sequentially 
apply (\ref{NoTorsion3}), (\ref{Frame Heat}),
(\ref{F Def}), and (\ref{NoTorsion3}):
\begin{align*}
\partial_s \psi_\alpha
&= D_\alpha \psi_s \\
&= D_\alpha D_i \psi_i \\
&= D_i D_\alpha \psi_i + F_{\alpha i} \\
&= D_i D_i \psi_\alpha + F_{\alpha i}.
\end{align*}
\end{proof}

\begin{proposition}(Connection bounds).
Let $\phi$ be a heat flow with classical initial data whose energy $E_0$ is less than $\Ec$.  
Let $e$ be a caloric gauge for $\phi$, and let $A_x$ denote the connection fields.  
Then we have the pointwise bounds
\begin{align}
\lVert \partial_x^k A_x(s) \rVert_{L_x^\infty(\rr)} &\lesssim_{E_0, k} s^{-(k+1)/2}, \label{A Infty} \\
\lVert \partial_x^k A_x(s) \rVert_{L_x^2(\rr)} &\lesssim_{E_0, k} s^{-k/2} \label{A L2}
\end{align}
for all $k \geq 0$ and $s > 0$, and the integrated estimates
\begin{align}
\int_0^\infty s^{(k-1)/2} \lVert \partial_x^k A_x(s) \rVert_{L_x^\infty(\rr)} \ds
&\lesssim_{E_0, k} 1, \label{Int A Infty}\\
\int_0^\infty s^{(k-1)/2} \lVert \partial_x^{k+1} A_x(s) \rVert_{L_x^2(\rr)} \ds
&\lesssim_{E_0, k} 1\label{Int A L2}
\end{align}
for all $k \geq 0$.
\end{proposition}
\begin{proof}
From (\ref{A Heat Evolution}) we have the integral representation
\begin{equation*}
A_x(s) = \int_s^\infty F_{sx}(s^\prime) \ds^\prime.
\end{equation*}
Differentiating covariantly as in (\ref{DA Decay}), we get
\begin{equation}
\lvert D_x^k A_x(s) \rvert
\lesssim
\int_s^\infty 
\sum_{ \substack{j_1 + j_2 + \cdots j_\ell = k + 3 \\ 
j_1, j_2, \ldots, j_\ell \geq 0 \\
\ell = 2, \ldots, k + 2} }
\e_{j_1}^{1/2}(s^\prime)
\cdots
\e_{j_\ell}^{1/2}(s^\prime) 
\ds^\prime,
\label{DA Integral Bound}
\end{equation}
and so applying (\ref{ek3}) yields
\begin{equation*}
\lVert D_x^k A_x(s) \rVert_{L_x^\infty(\rr)}
\lesssim_k
s^{-(k + 1)/2}
\end{equation*}
for $k \geq 0$.
Together (\ref{Matrix Action}) and an inductive argument prove (\ref{A Infty}).

Applying Minkowski's and H\"older's inequalities
in (\ref{DA Integral Bound}), we arrive at 
\begin{align}
\lVert D_x^k A_x(s) \rVert_{L_x^2}
\lesssim
\int_s^\infty
\sum_{ \substack{j_1 + j_2 + \cdots j_\ell = k \\ 
j_1, j_2, \ldots, j_\ell \geq 0 \\
\ell = 2, \ldots, k + 2} }
&\lVert \e_{j_1 + 2}^{1/2}(s^\prime) \rVert_{L_x^2} 
\lVert \e_{j_2 + 1}^{1/2}(s^\prime) \rVert_{L_x^\infty}
\nonumber \\
&\times
\lVert \e_{j_3}^{1/2}(s^\prime) \rVert_{L_x^\infty}
\cdots
\lVert \e_{j_\ell}^{1/2}(s^\prime) \rVert_{L_x^\infty} \ds^\prime.
\label{DA L2}
\end{align}
Using (\ref{ek1}), (\ref{ek4}), and Cauchy-Schwarz leads to (\ref{A L2})
with ordinary derivatives $\partial_x$ replaced by covariant derivatives $D_x$.
Ordinary derivatives may be recovered using (\ref{Matrix Action}) and (\ref{A Infty}).

Using (\ref{DA L2}) with $k+1$ in place of $k$ and applying the arithmetic- geometric-mean
inequality yields
\begin{align*}
\lVert D_x^{k+1} A_x(s) \rVert_{L_x^2}
\lesssim
&\sum_{ \substack{j_1 + j_2 + \cdots j_\ell = k \\ 
j_1, j_2, \ldots, j_\ell \geq 0 \\
\ell = 2, \ldots, k + 2} }
\int_s^\infty 
(s^\prime)^{j_1 - 1/2}
\lVert \e_{j+2}^{1/2}(s^\prime) \rVert_{L_x^2}^2 \\
&+ (s^\prime)^{-j_1 + 1/2} 
\lVert \e_{j_2+1}^{1/2}(s^\prime) \rVert_{L_x^\infty}^2
\lVert \e_{j_3}^{1/2}(s^\prime) \rVert_{L_x^\infty}^2 
\cdots
\lVert \e_{j_\ell}^{1/2}(s^\prime) \rVert_{L_x^\infty}^2 \ds^\prime,
\end{align*}
and hence by using (\ref{ek1}) and Fubini on the first term of the right hand side
and (\ref{ek3}), (\ref{ek4}), and Fubini on the second it follows that inequality
(\ref{Int A L2}) holds with ordinary derivatives replaced by covariant derivatives.
To recover ordinary derivatives in (\ref{Int A L2}), again use (\ref{Matrix Action}) and (\ref{A Infty}).

Finally, note that if $k \geq 1$, then applying the Gagliardo-Nirenberg inequality
(\ref{GN2}) to (\ref{Int A L2}) gives (\ref{Int A Infty}).  Therefore it only remains to check
the $k = 0$ case:
Using (\ref{DA Integral Bound}) and Fubini yields
\begin{equation*}
\int_0^\infty s^{-1/2} \lVert \partial_x A_x(s) \rVert_{L_x^2} \ds
\lesssim
\int_0^\infty s^{-1/2} \lVert \e_2^{1/2}(s) \rVert_{L_x^\infty}
\lVert \e_1^{1/2}(s) \rVert_{L_x^\infty} \ds,
\end{equation*}
and so the inequality follows from Cauchy-Schwarz and (\ref{ek4}).
\end{proof}

\begin{corollary}
Let $\phi$ be a heat flow with classical initial data with energy $E_0$ 
less than $\Ec$.  Let $e$ be
a caloric gauge for $\phi$. Then
\begin{align*}
\int_0^\infty s^{k-1} \lVert \partial_x^k \psi_x \rVert_{L_x^2(\rr)}^2 \ds
&\lesssim_{E_0, k} 1, \\
\sup_{s > 0} s^{(k-1)/2} \lVert \partial_x^{k-1} \psi_x \rVert_{L_x^2(\rr)}
&\lesssim_{E_0, k} 1, \\
\int_0^\infty s^{k-1} \lVert \partial_x^{k-1} \psi_x \rVert_{L_x^\infty(\rr)}^2 \ds
&\lesssim_{E_0, k} 1, \\
\sup_{s > 0} s^{k/2} \lVert \partial_x^{k-1} \psi_x \rVert_{L_x^\infty(\rr)}
&\lesssim_{E_0, k} 1
\end{align*}
for all $k \geq 1$.  Analogous estimates hold if one replaces $\partial_x \psi_x$ with $\psi_s$,
$\partial_x^2$ with $\partial_s$ and/or $\partial_x$ with $D_x$.
\label{Field Cor}
\end{corollary}
\begin{proof}
This follows from Theorem \ref{Main Energy Estimates}, (\ref{A Infty}), and
writing $\partial_x = D_x - A_x$.
\end{proof}

\subsection{The dynamic caloric gauge}

\begin{thm}(Dynamic caloric gauge).
Let $I$ be a time interval, let $\phi : \rr \times I \to \cM$ 
be a smooth map
Schwartz on each time slice with respect
to the point $\phi(\infty) \in \cM$,
and
let $e(\infty) \in \mathrm{ Frame}(T_{\phi(\infty)} \cM)$ be a frame for $\phi(\infty)$.  
Assume moreover that there is some $\epsilon > 0$ so that $E(\phi(t)) < \Ec - \epsilon$
for all $t \in I$.

Then we have the following conclusions.
The map $\phi$ extends smoothly to a dynamic heat flow
$\phi : \mathbf{R}^+ \times \rr \times I \to \cM$ and converges to $\phi(\infty)$ in
$C_\mathrm{loc}^\infty(\rr \times I)$ as $s \to \infty$.
There exists a unique smooth frame
$e \in \Gamma(\mathrm{Frame}(\phi^* T \cM))$ such that $e(t)$ is a caloric gauge for $\phi(t)$
that equals $e(\infty)$ at infinity for each $t \in I$.  
The map $\phi$ is Schwartz with respect to $\phi(\infty)$ and
all derivatives of $\phi$
in $s, x, t$ are Schwartz in $x$ for each fixed $s, t$.
Moreover, the time derivative field $\psi_t$ obeys
\begin{equation}
\partial_s \psi_t = D_i D_i \psi_t + F_{ti}
\label{EOM3t}
\end{equation}
and the time connection field $A_t$ obeys
\begin{equation}
\partial_s A_t = F_{ts}.
\label{EOM2t}
\end{equation}
\label{Dynamic CalGauge}
\end{thm}
\begin{proof}
The proof of uniqueness is as in Theorem \ref{CalGauge EU}.
Without loss of generality we take $I$ compact.
For each fixed $t$ we can extend the initial data $\phi(\cdot, t) : \rr \to \cM$ to
a smooth heat flow $\phi(\cdot, \cdot, t) : \mathbf{R}^+ \times \rr \to \cM$
by Theorem \ref{Main Energy Estimates}.  
As the global solution provided by Theorem \ref{Main Energy Estimates} coincides
with that directly provided locally by Picard iteration, the extension $\phi$
depends smoothly on the initial data since the nonlinearity being iterated
is a smooth function of $\phi$.
Noting that the constants in the bounds
in Theorem \ref{Main Energy Estimates} depend only upon the initial energy of the map
and recalling that we
assume $E(\phi(0, \cdot , t))$ is less than and bounded away from $\Ec$ uniformly
in $t$, we have that
$t \mapsto \phi(\cdot, \cdot, t)$ is locally smooth in smooth
topologies, and therefore the fixed-time heat flows
$\phi(\cdot, \cdot, t)$ can be joined together to create a smooth dynamic heat flow
$\phi : \mathbf{R}^+ \times \rr \times I \to \cM$.
Uniform bounds in the limit follow from the bounds in Theorem \ref{Main Energy Estimates}.
We note that in order to show that derivatives of $\phi$ involving the $t$ variable are
Schwartz in $x$ for fixed $s, t$, it suffices to directly apply a Picard iteration scheme as
in Theorem \ref{LWP}.

The rest of the construction of the caloric gauge follows that in Theorem \ref{CalGauge EU},
except that now we must take into account the dynamic variable $t$ and associated fields
$\psi_t, A_t$.  In particular,
the proofs of (\ref{EOM3t}) and (\ref{EOM2t}) are analogous to those of (\ref{EOM3}) and
(\ref{EOM2}) of Lemma \ref{EquationsOfMotion}.
\end{proof}

\nline
\emph{Acknowledgments.}
The author is indebted to Terence Tao for his encouragement, support, and guidance, and
thanks Peter Petersen for his teaching and discussions.  The author would also like
to thank the anonymous referees for corrections and suggestions, and for pointing out in an earlier
draft the need for an additional argument to handle the noncompact case.

\appendix

\section{Heat flow into noncompact manifolds} \label{Appendix Noncompact}
We claimed that the
caloric gauge construction is valid for any bounded geometry manifold
$\cM$, even though for convenience we made
an additional technical assumption.
In particular, we assumed that $\cM$ was a closed manifold, 
i.e., compact and without boundary.
Here we discuss how to remove this technical assumption.

First let us introduce an intrinsic notion of a Schwartz function
from $\mathbf{R}^2 \to \cM$.  As discussed in
\S \ref{S:Introduction}, we must choose a base point $p \in \cM$
to which our functions shall decay since there is no such natural choice on
a general manifold; however, once this point is selected, 
there is the following natural intrinsic definition of a Schwartz function.
\begin{definition}
Let $\cM$ be a manifold of bounded geometry and let $p \in \cM$.
A function $\phi : \mathbf{R}^2 \to \cM$ is said to be \emph{Schwartz
with respect to $p$} provided $\phi \in C^\infty(\mathbf{R}^2 \to \cM)$
and provided $\phi$ satisfies the decay conditions
\begin{equation*}
d_{\cM}( \phi(x), p ) = O_{\phi, N}( \langle x \rangle^{-N} )
\end{equation*}
for all positive integers $N$ and
\begin{equation*}
(\phi^* \nabla)_x^{\alpha} \phi(x) = O_{\phi, \alpha, N}( \langle x \rangle^{-N} )
\end{equation*}
for all (nonzero) multiindices $\alpha$ and positive integers $N$.
\end{definition}
When $\cM$ is a compact manifold, this definition is equivalent to that given
in \S \ref{S:Introduction}.  This is shown with a direct argument 
by choosing an smooth
isometric embedding $\iota: \cM \hookrightarrow \mathbf{R}^n$.

Equipped with a suitable definition of Schwartz function, we now give the underlying
idea in moving from the compact case to the noncompact case, due
to \cite{LiTa91}.
Start with Schwartz initial data $\phi_0 : \mathbf{R}^2 \to \cM$.  In view
of the fast-decay of $\phi_0$ and the fact that it is smooth, 
it follows that the image of $\mathbf{R}^2$
under $\phi_0$ lies within a compact set of $\cM$.  One then sets up an
embedding into a Euclidean space $\mathbf{R}^n$ to show local existence
and proves that the heat flow $\phi$ generated from the initial data $\phi_0$
stays within a compact subset of $\cM$ over its entire (maximal) interval of existence.
The rigorous argument is found in \cite[\S 3]{LiTa91}.
A posteriori, we choose a compact subset of $\cM$, choose a smooth isometric
embedding into some Euclidean space $\mathbf{R}^n$,
and apply the Duhamel formula (\ref{heatflow Duhamel}) as in
\S \ref{S:Small Energy Estimates} in order to establish that the heat flow $\phi$ is extrinsically 
Schwartz.  That the heat flow is also intrinsically Schwartz then follows from the fact that
the flow has a precompact image in $\cM$.
The blowup criterion derived from the proof is the same as that for the compact case:
the flow may be continued so long as $\partial_x \phi$ remains bounded.

We therefore have an entirely intrinsic analogue of Theorem \ref{LWP}, 
and, as indicated in \cite[Remark 3.1]{LiTa91}, we know that there is a compact set
in $\cM$ in which the heat flow $\phi$ always remains.

We now discuss where, in addition to Theorem \ref{LWP},
compactness was exploited.
Corollary \ref{Convergence Cor} employs compactness explicitly 
in a very mild way, namely, to establish that smooth solutions are Schwartz.  
Based upon the proof of the corollary, it is clear that the intrinsic version of 
Theorem \ref{LWP} suffices for this application.

A more subtle, implicit use of compactness is made in 
\S \ref{S:Minimal blowup solutions}, where the minimal blowup
solution argument appeals to the energy quantization of harmonic maps
(in particular, the existence of the groundstate energy $\Ec$)
and to the bubbling profile result \cite{St85} of Struwe.

For noncompact target manifolds we slightly modify the usual definition of
groundstate energy $\Ec$.
Given a compact subset
$K$ of $\cM$, we may smoothly extend $K$ to a larger compact (Riemannian) 
manifold $\widetilde{K}$ without boundary, i.e., a closed manifold.  Associate to
the manifold $\widetilde{K}$ its groundstate energy $\Ec(\widetilde{K})$
(as defined in \S \ref{S:Introduction}).
Now define the
groundstate energy of $K$ to be the supremum of energies $\Ec(\widetilde{K})$
taken over all smooth, closed extensions $\tilde{K}$ of $K$.  Finally, set
\begin{equation*}
\Ec(\cM) := \limsup_{n \to \infty} \Ec(K_n),
\end{equation*}
where $K_n$ denotes a (nested) compact exhaustion of $\cM$, i.e.,
each $K_n$ is compact, $K_{n} \subset K_{n+1}$, and
$\cup_n K_n = \cM$.
To show that $\Ec(\cM)$ is well-defined, one may suitably interlace two
given compact exhaustions $K_n$, $J_m$.

\begin{lemma}
For any dimension $m$ we have
\begin{equation*}
\Ec(\mathbf{H}^m) = +\infty,
\end{equation*}
i.e., the ground state energy of hyperbolic space is infinite.
\end{lemma}
\begin{proof}
Any compact $K \subset \mathbf{H}^m$ can be smoothly extended to a closed manifold
$\tilde{K}$
by means of taking the quotient of $\mathbf{H}^m$ by a suitable discrete group action.
However, as shown in \cite[\S 11]{EeSa64},
finite energy harmonic maps from $\mathbf{R}^2 \to \mathbf{H}^m$
do not exist.  Therefore
$\Ec(\tilde{K}) = +\infty$.
\end{proof}

Now, given Schwartz initial data $\phi_0 : \mathbf{R}^2 \to \cM$ with $E(\phi_0) < \Ec(\cM)$,
we may choose a sufficiently large compact subset $K$ of $\cM$ and suitable
smooth closed manifold extension $\tilde{K}$ of $K$ so that
$E(\phi_0) < \Ec(\tilde{K})$ 
and $\phi : I \times \mathbf{R}^2 \to \cM$ remains within $K$ for
its entire interval of existence $I$.
Replacing $\cM$ with $\tilde{K}$, i.e., considering now $\phi$ as a heat flow into
$\tilde{K}$, we may proceed with the arguments of \S \ref{S:Minimal blowup solutions}
verbatim.\footnote{As an alternative, one may proceed with a more direct argument,
only modifying the relevant parts of proofs as needed; the complete the details
of this approach, however, seem to involve taking a long tour through the extensive
literature on harmonic maps and harmonic map heat flow, most of which
assumes the target manifold to be a closed manifold.}

Finally, there are two additional places that make some appeal to compactness.
Theorem \ref{CalGauge EU} uses compactness
in order to show local uniform boundedness of the frame $e$ and 
to show integrability of 
$\lVert \partial_s \phi(s) \rVert_{L^\infty_x}$.  
In view of the result of \S \ref{S:Minimal blowup solutions}, though,
solutions are global, Schwartz, and with precompact image in $\cM$:
therefore we may still use an explicit embedding, and
all arguments carry over without modification.
Theorem \ref{Dynamic CalGauge} employs compactness in a similar way;
however, care must be taken if one desires to move from a local dynamic time
interval to a global one.  In particular, in order to attain
uniform bounds, one must ensure that the particular
dynamics in play keep $\phi$ within a bounded set.

\section{An energy space}
We propose an intrinsically defined energy space.  As an application
we have in view the study of the Schr\"odinger map initial value problem
\begin{equation*}
\begin{cases}
\partial_t \phi &= J(\phi) (\phi^* \nabla)_j \partial_j \phi \\
\phi(0) &= \phi_0
\end{cases}
\end{equation*}
on a K\"ahler manifold $\cM$ with complex structure $J$.  See for instance
\cite{Mc07} for definitions.
Our approach is motivated by that pursued by Tao in \cite{T4} in the study of wave maps.
The present setting is slightly simpler as compared to the wave maps setting in the sense
that our Cauchy data includes only an initial position $\phi_0$ as opposed to both
an initial position $\phi_0$ and initial velocity $\partial_t \phi_0$.

Let $\cL$ denote the Hilbert space of pairs $(\psi_s, \psi_x)$ of measurable functions
$\psi_s : \mathbf{R}^+ \times \rr \to \mathbf{R}^m$ and
$\psi_x : \mathbf{R}^+ \times \rr \to \mathbf{R}^m \times \mathbf{R}^m$ whose norm
\begin{equation}
\lVert (\psi_s, \psi_x) \rVert_{\cL}
:=
\frac{1}{2} \int_0^\infty \int_\rr \lvert \psi_s \rvert^2 \dx ds
+
\frac{1}{4} \int_\rr \lvert \psi_x \rvert^2 \dx
\label{L norm def}
\end{equation}
is finite.
Note that the orthogonal group $SO(m)$ acts unitarily on $\mathbf{R}^m$ and so naturally on $\cL$
unitarily.
Quotienting out by this compact group, we obtain a metric space $SO(m) \backslash \cL$.

Given classical initial data $\phi_0$ with energy $E_0 < \Ec$, consider
its heat flow extension $\phi(s, x)$ from $\rr$ to $\mathbf{R}^+ \times \rr$.
Choose a caloric gauge
$e$ for $\phi_0$ that equals some fixed frame $e(\infty) \in \mathrm{ Frame}(T_{\phi_0(\infty)} \cM)$
at infinity.
Let $S_p$ denote the set of functions $\rr \to \cM$ 
differing from $p \in \cM$ by a Schwartz function.
Hence $\phi_0 \in S_p$ with $p := \phi_0(\infty) \in \cM$.

Define the \emph{nonlinear Littlewood-Paley resolution map} 
$\jmath : \cS_p \to SO(m) \backslash \cL$ by
\begin{equation*}
\jmath(\phi_0) := SO(m)(\psi_s, \psi_x).
\end{equation*}
Rotating the frame $e(\infty)$ rotates the fields $\psi_s$, $\psi_x$ by an element of $SO(m)$,
and thus the only arbitrary choice we have made is $p = \phi_0(\infty)$.
We define the energy space $\ES_p$ to be the closure of $\jmath(\cS_p)$ in
$SO(m) \backslash \cL$.

\begin{lemma}(Energy identity).
For any $\phi \in \cS_p$ we have
\begin{equation}
E(\phi_0) = d_{SO(m) \backslash \cL}(j(\phi_0), 0).
\label{ES energy identity}
\end{equation}
\end{lemma}
\begin{proof}
The left hand side of (\ref{ES energy identity}) may be written as
\begin{equation*}
\frac{1}{2} \int_\rr \lvert \psi_x(0, \cdot) \rvert_{\phi_0^*h}^2 \dx.
\end{equation*}
Applying (\ref{NoTorsion3}) yields
\begin{equation*}
\partial_s \lvert \psi_x \rvert^2 = 2 \psi_j \cdot D_j \psi_s
\end{equation*}
and thus by integrating by parts and using (\ref{Frame Heat})
we obtain
\begin{align}
\partial_s \frac{1}{2} \int_\rr \lvert \psi_x(s, \cdot) \rvert^2 \dx
&=
-  \int_\rr D_j \psi_j \cdot \psi_s \dx \nonumber \\
&=
-  \int_\rr \lvert \psi_s \rvert^2 \dx. \label{ftc app}
\end{align}
As $E(\phi_0(s)) \to 0$ as $s \to \infty$ due to the energy monotonicity of the heat flow
and the fact that $E_0 < \Ec$, 
we conclude from (\ref{ftc app}) that
\begin{equation*}
E(\phi_0) =
\frac{1}{2} \int_0^\infty \int_\rr \lvert \psi_s \rvert^2 \dx ds
+
\frac{1}{4} \int_\rr \lvert \psi_x \rvert^2 \dx,
\end{equation*}
which, in view of (\ref{L norm def}), proves the lemma.
\end{proof}

\begin{thm}(Energy space).
Let $p:= \phi_0(\infty) \in \cM$ be fixed.  Then there exists a complete metric space
$\ES_p$ and a continuous map 
$\jmath : \cS_{p} \to \ES_{p}$ with the following properties:
\begin{enumerate}
\item[(i)] $\jmath(\cS_p)$ is dense in $\ES_p$.
\item[(ii)] $\jmath$ is injective.
\item[(iii)] Translation and dilation on $\cS_p$ extend to continuous isometric actions on $\ES_p$.
\item[(iv)] The energy functional $E : S_p \to \mathbf{R}^+$ extends to a continuous
functional $E : \ES_p \to \mathbf{R}^+$.
\item[(v)] If $\Phi \in \ES_p$ has energy $E(\Phi) = 0$, then $\Phi = \jmath(p)$.
\end{enumerate}
\label{Energy Space}
\end{thm}
\begin{proof}
Property (i) follows from construction.

To show (ii), suppose we have classical data $\phi_0, \tilde{\phi}_0 \in \cS_p$
with $\jmath (\phi_0) = \jmath (\tilde{\phi}_0)$.  Then there exist caloric gauges
$e, \tilde{e}$ with respect to which $\psi_s = \tilde{\psi}_s$ and $\psi_x = \tilde{\psi}_x$.
By construction $\phi_0(\infty) = p = \tilde{\phi}_0(\infty)$.  
If need be we apply a rotation
in $SO(m)$ so that $e(\infty) = \tilde{e}(\infty)$.
From (\ref{DF Def}) and (\ref{Frame evolution}) respectively we have that
\begin{equation*}
\partial_s \phi_0 = e \psi_s
\quad \text{and} \quad
(\phi_0 \nabla)_s e = 0,
\end{equation*}
and hence a system of ODEs from which we may recover $\phi_0$ and $e$ from the
boundary data $\phi_0(\infty)$ and $e(\infty)$, using (\ref{time deriV bound}) 
to justify integrability.

Property (iii) is an easy consequence of (i).

Property (iv) is also straightforward.  We note, however, that as in \cite{T4} one may develop
an energy space whose norm $\cL$ is given by $\int_0^\infty \int_\rr \lvert \psi_s \rvert^2 \dx ds$
rather than the one we adopted in (\ref{L norm def}).  In such a case property (iv)
requires verification and is not trivial.  We do not pursue this approach here.

Property (v) follows from (i) and (iv).

\end{proof}

\bibliography{CGGR}
\bibliographystyle{amsplain}

\end{document}